\documentclass[11pt, a4paper]{amsart}

\usepackage[foot]{amsaddr}

\usepackage{amsfonts,amssymb, url, amsthm, subcaption, hyperref, tikz, amsmath,
bigdelim, xcolor, appendix, environ, stmaryrd, mathabx, enumerate}

\usepackage{mathrsfs}

\usepackage[utf8]{inputenc}
\usepackage[T1]{fontenc}
\usepackage[top=3cm, bottom=3cm, left=3cm, right=3cm]{geometry}

\newtheorem{theorem}{Theorem}[section]
\newtheorem{lemma}[theorem]{Lemma}
\newtheorem{proposition}[theorem]{Proposition}
\newtheorem{corollary}[theorem]{Corollary}
\newtheorem{definition}[theorem]{Definition}
\newtheorem{assumption}[theorem]{Assumption}
\newtheorem{remark}[theorem]{Remark}
\newtheorem{notation}[theorem]{Notation}
\theoremstyle{definition}

\newcommand{\RR}{\mathbb{R}}
\newcommand{\NN}{\mathbb{N}}
\newcommand{\ZZ}{\mathbb{Z}}
\newcommand{\YY}{\mathbb{Y}}
\newcommand{\EE}{\mathbb{E}}

\newcommand{\PP}{\mathbb{P}}
\newcommand{\TT}{\mathbb{T}}

\newcommand{\mL}{\mathcal{L}}

\newcommand{\mP}{\mathcal{P}}

\newcommand{\mI}{\mathcal{I}}
\newcommand{\mY}{\mathcal{Y}}
\newcommand{\mF}{\mathcal{F}}
\newcommand{\mD}{\mathcal{D}}

\newcommand{\mS}{\mathcal{S}}
\newcommand{\mB}{\mathcal{B}}
\newcommand{\mA}{\mathcal{A}}

\newcommand{\reso}{\varodot}

\newcommand{\ve}{\varepsilon}

\newcommand{\vt}{\vartheta}

\newcommand{\bigslant}[2]{{\raisebox{.1em}{$#1$}\left/\raisebox{-.1em}{$#2$}\right.}}
\newcommand*{\ud}{\mathrm{\,d}}

\makeatletter
\newcommand\RSloop{\@ifnextchar\bgroup\RSloopa\RSloopb}
\makeatother
\newcommand\RSloopa[1]{\bgroup\RSloop#1\relax\egroup\RSloop}
\newcommand\RSloopb[1]%
  {\ifx\relax#1%
   \else
     \ifcsname RS:#1\endcsname
       \csname RS:#1\endcsname
     \else
       \GenericError{(RS)}{RS Error: operator #1 undefined}{}{}%
     \fi
   \expandafter\RSloop
   \fi
  }
\newcommand\X{0}
\newcommand\RS[1]%
  {\begin{tikzpicture}
     [every node/.style=
       {circle,draw,fill,minimum size=1.5pt,inner sep=0pt,outer sep=0pt},
      line cap=round
     ]
   \coordinate(\X) at (0,0);
   \RSloop{#1}\relax
   \end{tikzpicture}
  }
\newcommand\RSdef[1]{\expandafter\def\csname RS:#1\endcsname}
\newlength\RSu
\RSdef{i}{\draw (\X) -- +(90:\RSu) node{};}
\RSdef{l}{\draw (\X) -- +(110:\RSu) node{};}
\RSdef{r}{\draw (\X) -- +(70:\RSu) node{};}
\RSdef{I}{\draw (\X) -- +(90:\RSu) coordinate(\X I);\edef\X{\X I}}
\RSdef{L}{\draw (\X) -- +(120:\RSu) coordinate(\X L);\edef\X{\X L}}
\RSdef{R}{\draw (\X) -- +(60:\RSu) coordinate(\X R);\edef\X{\X R}}

\newcommand{\TA}{\scalebox{0.8}{\RS{lr}}}
\newcommand{\TB}{\scalebox{0.8}{\RS{rLrl}}}
\newcommand{\TC}{\scalebox{0.8}{\RS{rLrLrl}}}
\newcommand{\TD}{\scalebox{0.8}{\RS{{Lrl}{Rrl}}}}

\begin{document}

\title{Synchronization for KPZ}

\author{Tommaso Cornelis Rosati}
\address{Humboldt-Universit\"at zu Berlin}

\thanks{This paper was developed within the scope of the IRTG 1740 /
TRP 2015/50122-0, funded by the DFG / FAPESP}



\begin{abstract} 
  We study the longtime behavior of KPZ-like
  equations:
  \[ 
    \partial_{t}h(t,x) = \Delta_{x} h (t, x) + | \nabla_{x}h
    (t,x)|^{2} + \eta(t, x), \qquad h(0, x) = h_0(x), \qquad (t, x) \in (0,
    \infty) \times \TT^{d}, 
  \] 
  on the \(d-\)dimensional torus \(\TT^{d}\) driven by an ergodic noise
  \(\eta\) (e.g. space-time white in \(d= 1\)). The analysis builds
  on infinite-dimensional extensions of similar results for positive random
  matrices. We establish a one force, one solution principle and derive almost
  sure synchronization with exponential deterministic speed in appropriate H\"older spaces.
\end{abstract}

\maketitle
\setcounter{tocdepth}{1} 

\section*{Introduction}

The present article concerns the study of stochastic partial differential
equations (SPDEs) of the form:
\begin{equation}\label{eqn:intro_kpz}
  (\partial_{t} {-} \Delta_{x})h(t,x)= | \nabla_{x} h|^{2} (t, x) +
  \eta(t, x), \qquad h(0, x) = h_0(x), \qquad (t, x) \in
  (0, \infty) \times \TT^{d}, 
\end{equation}
where \(\eta\) is a random forcing and \(\TT^{d}\) the $d$-dimensional torus.
In dimension \(d=1\) with \(\eta\) space-time white noise the above is
known as the Kardar-Parisi-Zhang (KPZ) equation, and is a renowned model in the study of growing
interfaces. It is related to a universality conjecture
\cite{QuastelSpohn2015KpzUniversality}, according to which many asymmetric
growth models are described, on large spatial and temporal scales, by a common
universal object: the so-called KPZ fixed point. The KPZ equation is itself the
scaling limit of many such processes in a \textit{weakly asymmetric} regime
and it is expected to converge to the mentioned fixed point on large
scales. 

This picture motivates in part interest behind the longtime behavior
of equations of type \eqref{eqn:intro_kpz}. Another motivation comes
Burgers'-like equations. These are toy models in fluid dynamics, and are
formally linked to KPZ by $v = \nabla_x h$:
\begin{equation}\label{eqn:intro_burgers}
  (\partial_{t} {-} \Delta_{x})v(t, x) = \nabla_{x}|v|^{2}(t, x) +
  \nabla_{x} \eta(t, x), \quad v(0,x) = v_0(x), \quad (t,x) \in
  (0, \infty) \times \TT^{d}.
\end{equation}

In the case of the original KPZ
equation, wellposedness 
\cite{Hairer2013SolvingKPZ, Hairer2014, GubinelliImkellerPerkowski2015} was a
milestone.
Preceding these results there was no clear understanding of the quadratic nonlinearity in
\eqref{eqn:intro_kpz}, yet the equation could be studied through
the Cole-Hopf transform, by imposing that \(u = \exp (h)\) solves the linear
stochastic heat equation (SHE) with multiplicative noise (a step that can be made rigorous for 
  smooth \(\eta\) but requires particular care and the introduction of
renormalization constants if \(\eta\) is space-time white noise):
\begin{equation}\label{eqn:intro_she}
  (\partial_{t} {-} \Delta_{x})u(t, x) = \eta (t, x) u (t, x), \qquad u(0,x) =
  u_0(x), \qquad(t, x) \in (0, \infty) \times \TT^{d}.
\end{equation}
In addition to proving wellposedness of the KPZ equation, \cite{Hairer2014}
introduces the notion of \textit{subcriticality}, which provides a formal condition on
\(\eta\) under which Equations~\eqref{eqn:intro_kpz} and \eqref{eqn:intro_she} remain well-posed.
Recent works show that this condition is indeed sufficient
  \cite{BrunedHairerZambotti2019AlgebraicRenormalisation,
  BrunedChandraChevyrevHairer2017RenormalisinSPDEs,
ChandraHairer2016AnalyticBPHZ}.

In consideration of these results, we will not discuss
questions regarding the wellposedness of~\eqref{eqn:intro_kpz}. Instead, our aim is
to prove results concerning the longtime behavior
of solutions, under the assumption that a solution map to
the SHE~\eqref{eqn:intro_she} is given and satisfies some natural requirements
that will be introduced below.

The linearization of \eqref{eqn:intro_kpz} provides additional
structure to the equation that allows to prove strong ergodic
properties. The first simple observation is that
\eqref{eqn:intro_kpz} is shift invariant, meaning that any ergodic
property will be proved either ``modulo constants'', that is identifying two
functions \(h, h^{\prime} : [0, \infty) \times \TT^{d} \to \RR\) if there exists
a \(c \colon [0, \infty) \to \RR\) such that \(h(t,x) -
h^{'}(t,x) = c(t), \forall t,x\), or for the gradient \(v = \nabla_{x} h\),
which satisfies Burgers' equation.

On one hand, unique ergodicity of the KPZ Equation ``modulo constants'' was
first established in \cite{HairerMattingly2018Feller} as a
consequence of a strong Feller property that holds for a wide class of SPDEs.
Moreover, for the KPZ equation the invariant measure
is known to be the Brownian bridge \cite{FunakiQuastel2015KPZInvariant}. In
addition, \cite{GubinelliPerkowski2018GeneratorBurgers} proves a spectral gap for
Burgers' equation, implying
exponential convergence to the invariant measure, although the article
restricts to initial conditions that are ``near-stationary''.

On the other hand, \cite{Sinai1991Buergers} considers the noise
\(\eta(t, x) = V(x) \partial_{t} \beta (t)\) for \(V \in C^{\infty}(\TT^{d})\) and a
Brownian motion \(\beta\). The article shows that there exists a random function
\( \overline{v}(t,x)\) defined for all \(t \in \RR\) such that almost surely,
independently of the initial conditions \(v_0\) within a certain class:
\[ \lim_{t \to \infty}  v(t, x) - \overline{v}(t, x) = 0,\] 
for all \(x \in \TT^{d}\) and with \(v\) solving~\eqref{eqn:intro_burgers}.
This property is referred to as \textit{synchronization}. 
In addition, if one starts Burgers' equation at time \({-} n\) with \(v^{{-} n}( {-} n, x) =
v_{0}(x)\):
\[ \lim_{n \to \infty} v^{{-} n}(t, x) = \overline{v}(t, x), \qquad \forall (t, x) \in
({-} \infty, \infty) \times \TT^{d}.\] 
This property is called a \textit{one force, one solution} principle (1F1S) and it
implies that \(\overline{v}\) is the unique (ergodic)
solution to \eqref{eqn:intro_burgers} on \(\RR\). Results of this
kind have subsequently been generalized in many directions, most notably to the
inviscid case 
\cite{WeinanKhaninSinai2000BurgersInviscous} or to infinite volume, for
example in \cite{BakhtinCatorKhanin2014Burgers} and
recently in \cite{DunlapGrahamRyzhik2019stationary}, for specific classes
of noises. 

We will attempt to understand and extend the results in
\cite{Sinai1991Buergers}
through an application of the theory of random dynamical systems. The power of
this approach lies in the capacity of treating any noise \(\eta\) such that:
\begin{enumerate}[(i)]
  \item The noise \(\eta\) is ergodic: see Proposition
    \ref{prop:mixing_conditions} for a classical condition if \(\eta\) is
    Gaussian.

  \item Equation \eqref{eqn:intro_she} is almost surely well-posed: there exists
    a unique, global in time solution for every \(u_{0} \in C(\TT^{d})\), the
    solution map being a linear, compact, strictly positive operator on
    \(C(\TT^{d})\).
\end{enumerate}
In particular, \(\eta\) can be chosen
to be space-time white noise or a noise that is fractional in time.

In the original work \cite{Sinai1991Buergers},
the solution \(u\) to \eqref{eqn:intro_she} evaluated at time \(n\) is
represented by \(u(n, x) =
A^{n} u_{0} (x)\) for a compact strictly positive operator \(A^{n}\). The proof of the
result makes use in turn of the explicit representation of the operator
\(A^{n}\) via the Feynman-Kac formula. Such representation becomes more technical when the noise \(\eta\) is
not smooth and requires some understanding of random polymers (cf. 
\cite{CannizzaroChouk2018Multidimensional, DelarueDiel2016OneDimensionalSDE}
for the case of space-time white noise). 

In this work we will avoid the
language of random polymers.
If \(\eta\) were a time-independent noise, the synchronization of the solution
\(v\) to \eqref{eqn:intro_burgers} would
amount to the convergence, upon rescaling, of \(u\) to the random eigenfunction of
\(A^{1}\) associated to its largest eigenvalue: an instance of the Krein-Rutman
Theorem. We will extend this argument to the non-static case with
an application of the theory of random dynamical systems. The key tool is a
contraction principle for
positive operators in projective spaces under Hilbert's projective metric (see
\cite{Bushell1973} for an overview). Such method was already deployed in \cite{ArnoldGundlachDemetrius1994} and later
refined by \cite{Hennion1997} in the study of random matrices. Their proofs
naturally extend to the infinite-dimensional case, giving rise to an ergodic
version of the Krein-Rutman theorem: see Theorem \ref{thm:random_krein_rutman}. 

In this way one obtains exponential synchronization and
1F1S ``modulo constants'' for the KPZ equation: in an example with smooth noise we show
  that the constants can be chosen time-independent, a fact that we expect to
hold in general. The exponential speed is deterministic and related to the
contraction principle.

Some complications show up when proving
convergence in appropriate H\"older spaces, depending on the regularity of the
driving noise: see Theorem \ref{thm:synchronization_for_kpz}. This step
requires for example a bound on the average:
\[ \EE \sup_{x \in \TT} |h(t,x)|,\]
for fixed \(t >0\). In concrete
examples we show how a control on this term can be obtained from a quantitative version
of a strong maximum principle for \eqref{eqn:intro_she}: in particular the case
of space-time white noise requires some care and the proof makes use of the
path-wise solution theory to the equation. Although the study of convergence in H\"older
spaces seems to be new, for different reasons moment bounds of the likes of the one above appeared already in
\cite{ArnoldGundlachDemetrius1994}, but are simpler to
verify in the finite-dimensional case.

As we already mentioned, the examples we treat are the original KPZ equation, namely the case of
\(\eta\) being space-time white noise in \(d=1\), and the case of
\(\eta(t, x) = V(x) \ud \beta^{H}_{t}\) for \(\beta^{H}\) a fractional Brownian
motion of Hurst parameter \(H > \frac{1}{2}\) and \(V \in C^{\infty}(\TT)\). In the latter
case the solution is not Markovian, and ergodic results are rare, see for
example \cite{MaslowskiPospivsil2008ErgodicityFractional} for ergodicity of linear SPDEs
with additive fractional noise.

There are several instances of applications of the
theory of random dynamical systems to stochastic PDEs. Particularly related to
our work is the study of order-preserving systems which admit some random
attractor \cite{ArnoldChueshov1998OrderPreserving,
FlandoliGessScheutzow2017SynchroByNoiseOrderPreserving, ButkovskyScheutzow2019}.  The spirit of these results is
similar to ours. Yet, although the linearity of \eqref{eqn:intro_she} on one hand
guarantees order preservation, on the other hand it does not allow the existence
of a single random attractor. In this sense, our essentially linear case appears to be a
degenerate example of the synchronization addressed in the works above.

\proof[Acknowledgements] The author is very grateful to Nicolas Perkowski for
inspiring this work and providing numerous insights and helpful comments. Many
thanks also to Benjamin Gess for several interesting discussions, and to
the anonymous reviewer that has carefully read the work and suggested many
improvements.

\section{Notations}

Let \(\NN = \{1,2, \dots\}, \ \NN_{0} = \NN \cup \{0\}\), 
\(\mathbb{R}_{+} = [0, {+} \infty)\) and \(\iota = \sqrt{ - 1} \). Furthermore,
for \(d \in \NN\) let \(\TT^{d}\) be the torus \(\TT^{d} =
\bigslant{\RR^{d}}{\ZZ^{d}}\) (here \(\ZZ^{d}\) acts by translation on
\(\RR^{d}\)). The case \(d =1\) is of particular interest, so we write
\(\TT = \TT^{1}\).

For a general set \( \mathcal{X}\) and functions \(f, g  \colon \mathcal{X} \to
\RR\) write \(f \lesssim g\) if \( f(x) \leq C g(x)\) for all \(x \in
\mathcal{X}\) and a constant \(C>0\) independent of \(x\). To clarify on which
parameters \(C\) is allowed to depend we might add them as subscripts to the
``\(\lesssim\)'' sign.

For \(\alpha>0\) let \(\lfloor \alpha \rfloor\) be the smallest
integer beneath \(\alpha\) and for a multiindex \(k \in \NN^{d}\) write
\(|k| = \sum_{i = 1}^{d} k_i\).
Denote with \(C(\TT^{d})\) the space of continuous real-valued
functions on \(\TT^{d}\), and, for \(\alpha>0\), with \(C^{\alpha}(\TT)\) the
space of \( \lfloor \alpha \rfloor-\) differentiable functions \(f\) such that
\( \partial^{k} f\) is \((\alpha {-} \lfloor \alpha \rfloor)-\)H\"older
continuous for every multiindex \(k \in \NN^{d}\) such that \(|k| = \lfloor
\alpha \rfloor\), if \(\alpha {-} \lfloor \alpha \rfloor>0\), or simply
continuous if \(\alpha \in \NN_{0}\). For \(\alpha \in \RR_{+}\) we obtain the
following seminorms on \(C^{\alpha}(\TT^{d})\):
\[ 
  [ f]_{\alpha} = \max_{ |k|= \lfloor \alpha \rfloor} \ \ \| \partial^{k} f
  \|_{\infty} 1_{ \{|k|>0\}} + \sup_{x, y \in
  \TT^{d}}\frac{|\partial^{k} f(x) {-} \partial^{k} f(y)|}{|x {-} y|^{\alpha {-}
  \lfloor \alpha \rfloor}}. 
\] 
We write \(C^{\infty}(\TT^{d}) = \bigcap_{k \in \NN} C^{k}( \TT^{d})\).

Now, let \(X\) be a Banach space. We denote with \(\mB(X)\) the Borel
\(\sigma-\)algebra on \(X\). Let \([a,b] \subseteq \RR\) be an
interval, then define 
\(C([a,b]; X)\) the space of continuous functions \(f  \colon [a,b] \to
X\). For any \(O \subseteq \RR\), we write \(C_{\mathrm{loc}}(O; X)\) for the
space of continuous functions with the topology of uniform
convergence on all compact subsets of \(O\). Given two Banach spaces
\(X, Y\) denote with \(\mL(X ; Y)\) the space of linear bounded operators \(A \colon
X \to Y\) with the classical operator norm. If \(X=Y\) we write simply \(
\mL(X)\). 

Next we introduce Besov spaces. Following \cite[Section
2.2]{BahouriCheminDanchin2011FourierAndNonLinPDEs} choose a smooth
dyadic partition of the unity on \(\RR^{d}\) (resp. \(\RR^{d {+} 1}\)) \( (
\chi, \{\varrho_{j}\}_{j \geq 0})\)  and define \(\varrho_{{-} 1}
= \chi\) and define the Fourier transforms for \(f \colon \TT^{d} \to \RR\) and
\(g  \colon \RR \times \TT^{d} \to \RR\):
\begin{align*} 
  \mF_{\TT^{d}} f (k) &= \int_{\TT^{d}} e^{ {-} 2 \pi \iota \langle k, x \rangle}
  f(x) \ud x, \qquad k \in \ZZ^{d}, \\
  \mF_{\RR \times \TT^{d}} g (s, k) & = \int_{\RR \times \TT^{d}} \! \! \!
  e^{ {-} 2 \pi \iota
  (s t {+} \langle k, x \rangle )} g(t, x) \ud t \ud x, \qquad (s, k) \in
  \RR \times \ZZ^{d}.
\end{align*}
These definitions extend naturally to
spatial (resp.\ space-time) tempered distributions \(\mS^{ \prime}(\TT^{d}) \) (resp.
\(\mS^{\prime}(\RR \times\TT^{d})\)), which are the topological duals of Schwartz
functions: \(\mS(\TT^{d}) = C^{\infty}(\TT^{d})\) and
\[ \mS(\RR \times \TT^{d}) =
\Big\{ \varphi  \ \colon  \sup_{t \in \RR, x \in \TT^{d}} \big\{ (1+|t|)^{p}|
\partial_{x}^{\mu} \varphi (t,x)|\big\} < \infty, \ \forall p \geqslant 0, \mu \in
\NN_{0}^{d+1} \Big\}.\] 
Similarly one defines the respective inverse Fourier transforms
\(\mF_{\TT^{d}}^{-1}\) and \(\mF_{\RR \times \TT^{d}}^{-1}\). Then define the
spatial (resp.\ space-time) Paley blocks: 
\begin{align*}
  \Delta_{j} f(x) = \mF^{{-} 1}_{\TT^{d}}[ \varrho_{j} \cdot \mF_{\TT^{d}}f]
  (x), \qquad \Delta_{j}g (t, x) = \mF^{{-} 1}_{\RR \times \TT^{d}} [\varrho_{j}
  \cdot \mF_{\RR \times \TT^{d}} g](t,x).
\end{align*}
Eventually one defines, for \(\alpha \in \RR, \ a>0, \ p,q \in [1, \infty]\), the spaces \(B^{\alpha}_{p, q} (\TT^{d})\) and
\(B^{\alpha, a}_{p, q}(\RR \times \TT^{d})\) as the set of tempered distributions such
that, respectively, the following norms are finite:
\begin{align*} 
  \| f \|_{B^{\alpha}_{p, q}(\TT^{d})} & = \| (2^{j \alpha} \| \Delta_{j} f
  \|_{L^{p}(\TT) } )_{j \geq {-} 1}\|_{\ell^{q}}, \\
  \| g \|_{B^{\alpha, a }_{p, q} (\RR \times \TT^{d})} & = \| ( 2^{j \alpha} \| \Delta_{j}f( \cdot)/
  \langle \cdot \rangle^{a} \|_{L^{p}(\RR \times \TT^{d})})_{j \geq {-} 1}
  \|_{\ell^{q}}, 
\end{align*} 
where we denote with \(\langle (t,x) \rangle\) the weight \( \langle (t, x)
\rangle = 1 {+} |t|\). For \(p=q=2\) one obtains the Hilbert spaces \(H^{\alpha} (\TT^{d}) = B^{\alpha}_{2,
2}(\TT^{d})\) and 
\begin{equation}\label{eqn:definition-space-time-H-space}
H^{\alpha}_{a}(\RR \times \TT^{d})= B^{\alpha,
a}_{2,2}(\RR \times \TT^{d}).
\end{equation}
One can also consider functions that depend on time only and introduce, for the
same range of parameters, the spaces \(B^{\alpha, a}_{p, q}(\RR)\) via the norm:
\[\| f \|_{B^{\alpha, a }_{p, q} (\RR )} = \| ( 2^{j \alpha} \| \Delta_{j}f( \cdot)/
  \langle \cdot \rangle^{a} \|_{L^{p}(\RR)})_{j \geq {-} 1}
  \|_{\ell^{q}}, \qquad \langle t \rangle= 1+|t|. \]
Here the Paley blocks are defined by \(\Delta_{j} f(t) =
\mF_{\RR}^{-1}(\varrho_{j} \cdot \mF_{\RR} f)(t),\) for a dyadic partition of the
unity \(\{\varrho_{j}\}_{j \geqslant -1}\) on \(\RR\). As above we then define
\begin{equation}\label{eqn:definition-time-H-space}
\begin{aligned}
H^{\alpha}_{a}(\RR) = B^{\alpha, a}_{2,2}(\RR).
\end{aligned}
\end{equation}
Finally, recall that for \(p=q= \infty\) and \(\alpha \in
\RR_{+} \setminus \NN_{0}\): \(B^{\alpha}_{\infty, \infty}(\TT^{d}) =
C^{\alpha}(\TT^{d})\) (see e.g. \cite[Chapter 2]{Triebe2010}).

\section{Setting}

This section, based on \cite{Bushell1973}, introduces the projective space of positive continuous functions
and a related contraction principle for strictly positive operators. Let \(X\) be a Banach
space and \( K \subseteq X\) a closed cone such that \(K \cap ({-} K) =
\{0\} \). Denote with \(\mathring{K}\) the interior of \(K\)
and write \(K^{+} = K \setminus \{0\}\). Such cone induces a partial order in \(X\) by
defining for \(x, y \in X\): 
\[ 
  x \leq y \Leftrightarrow y {-} x \in K \ \ \text{ and } \ \
  x < y \Leftrightarrow y {-} x \in \mathring{K}.
\] 
Consider for \(x, y \in K^{+}\): 
\[ M(x, y) = \inf \{ \lambda \geq 0  \colon x \leq \lambda y\}, \qquad m(x, y)
  = \sup \{ \mu \geq 0 \colon \mu y \leq x\}, 
\] 
with the convention \(\inf \emptyset = \infty\). Then \(M(x, y) \in (0,
\infty]\) and \(m(x, y) \in [0, \infty)\) so that one can define Hilbert's
projective distance: 
\[ d_{H}(x, y) = \log{(M (x, y))} {-} \log{ ( m(x,y))} \in [0, \infty],
  \qquad \forall \ x, y \in K^{+}.
\] 
This metric is only semidefinite positive on \(K^{+}\), and may be infinite. A
remedy for the first issue is to consider an affine
space \(U \subseteq X\) which intersects transversely \(K^{+}\), that is: 
\[ \forall x \in K^{+}, \ \ \ \ \exists ! \lambda > 0
  \ \ \ \ \text{ s.t. } \ \ \ \ \lambda x \in U. 
\] 
Write \(\lambda(x)\) for the normalization constant above. As for the second
issue, one can observe that the distance is finite on the interior of
\(K\), cf. \cite[Theorem 2.1]{Bushell1973}, and thus, defining \(E =
\mathring{K} \cap U\), one has that \((E, d_{H})\) is a metric space.  

Consider now \(\mL(X)\) the set of linear bounded operators on
\(X\), and for an operator \(A \in \mL(X)\) the following conventions define
different concepts of positivity: 
\[ 
  \begin{aligned} A(K) & \subseteq K
    \qquad \Rightarrow \qquad A \text{ nonnegative. } \\
    A(\mathring{K}) & \subseteq \mathring{K} \qquad \Rightarrow \qquad A
    \text{ positive. } \\
    A(K^{+}) & \subseteq \mathring{K} \qquad \Rightarrow \qquad A \text{
    strictly positive. } 
  \end{aligned} 
\] 
The projective action of a positive operator \(A\) on \(X\) is then defined by:
\(A^{\pi} x = Ax / \lambda(A x)\). One can view \(A^{\pi}\) as a map
\(A^{\pi} \colon E \to E\) and one then denotes with \( \tau(A)\) the projective
norm associated to \(A\): 
\begin{equation}\label{eqn:birkhoff_contraction} 
  \tau(A) = \sup_{ \substack{x, y \in E \\ x \neq y}} \frac{d_{H}(A^{\pi} x, A^{\pi}
  y)}{d _{H} (x, y)}.  
\end{equation} 
The backbone of our approach is Birkhoff's theorem for positive operators
\cite[Theorem 3.2]{Bushell1973}, which is stated below.
\begin{theorem}\label{thm:Birkhoff_positive_operators}
  Let \(\Delta(F)\) denote the diameter of a set \(F \subseteq E\):
  \[ 
    \Delta(F) = \sup_{x, y \in F} \{ d_{H}(x,y)\}.
  \]
  The following identity holds:
\[
  \tau(A) = \mathrm{tanh} \Big( \frac{1}{4} \Delta (A^{\pi} (E)) \Big) \leq 1.
\] 
\end{theorem}
Then denote with \( \mL_{\mathrm{cp}}(X)\) the space of positive operators \(A\)
which are contractive in \((E, d_{H})\):
\[ 
  A \in \mL_{\mathrm{cp}}(X) \ \ \Leftrightarrow \ \ A \in \mL(X), \ \ A
  \text{ positive, } \ \  \tau(A) <1. 
\] 

The only example considered in this work is \(X = C(\TT^{d})\) the space of
real-valued continuous functions on the torus, where \(K\) is the cone of positive
functions. Here the following holds.

\begin{lemma}\label{lem:completeness_and_properties_positive_fcts}
  Let \(X = C(\TT^{d})\) and \(K = \{ f \in X \ \colon \ f(x) \geq 0, \ \  \forall x
  \in \TT^{d}\}\), and consider:
  \[ 
    U = \Big\{ f \in X \colon \int_{\TT^{d}} f(x)
    \ud x = 1 \Big\}.   
  \]
  For the associated metric space \((E, d_{H})\) the following inequality
  holds:
  \begin{equation}\label{eqn:log_hilbert_distance_equivalence} \| \log{ (f)} {-}
    \log{(g)} \|_{\infty} \leq d_{H} (f, g) \leq  2\| \log{ (f)} {-} \log{ (g)}
    \|_{\infty}, \qquad \forall f, g \in E.  
  \end{equation} 
  In particular, \((E, d_{H})\) is a \textit{complete} metric space. In
  addition, if a strictly positive operator \(A\) can be represented by a
  kernel, i.e.\ there exists \(K \in C(\TT^{d} \times \TT^{d})\) such that:
  \[ A (f) (x) = \int_{\TT^{d}} K(x,y) f(y) \ \ud y, \ \ \ \forall x \in
  \TT^{d} \] 
  and there exits constants \(0 < \alpha \leq \beta < \infty\) such that 
  \[ \alpha \leq K(x,y) \leq \beta, \ \ \ \forall x, y \in \TT^{d},\]
  then \(A\) is contractive, i.e. \(A \in \mL_{\mathrm{cp}} (X)\).
\end{lemma}

\begin{proof}
  As for the inequality, since \(f,g \in U\) (and hence \(\smallint f(x) \ud x
  = \smallint g(x) \ud x = 1\)), there exists a point \(x_{0}\) such that
  \(f(x_0) = g(x_0)\). In particular in the sum 
  \begin{align*} 
    \max \left( \log{ (f/g)} \right) {-} \min \left( \log{(f/g)}  \right) =
    \max \left( \log{ (f/g)} \right) {+} \max \left( \log{(g/f)}  \right)
  \end{align*} both terms are
  positive an bounded by \( \| \log{(f)} {-} \log{(g)} \|_{\infty}\).
  Conversely we have that: 
  \begin{align*} 
    \| \log{(f)} {-} \log{(g)}
    \|_{\infty} \leq \max \left( \log{(f)} {-} \log{(g)}  \right) + \max
    \left( \log{(g)} {-} \log{(f)}  \right).  
  \end{align*}

  Completeness of \( (E, d_{H})\) is a consequence of
  Inequality~\eqref{eqn:log_hilbert_distance_equivalence}: for a given
  Cauchy sequence \(f_{n} \in E\) the sequence \(\log (f_{n})\) is a Cauchy
  sequence in \(C(\TT^{d})\). By completeness of the latter there exists a
  \(g \in C(\TT^{d})\) such that \(\log (f_{n}) \to g\). By dominated
  convergence \(\exp (g) \in E\), and hence \(f_{n} \to \exp (g)\) in \(E\).

  The result regarding the kernel can be found in \cite[Section
  6]{Bushell1973}.
 
\end{proof}

\begin{remark}

  For the sake of simplicity we did not address the general
  question of completeness of the space \((E, d_H)\), since in the case of
  interest to us completeness follows from
  \eqref{eqn:log_hilbert_distance_equivalence}. Yet general criteria for
  completeness are known, see for example \cite[Section 4]{Bushell1973} and the
references therein.

\end{remark}
 
\begin{remark}

  In view of \eqref{eqn:birkhoff_contraction}, an application of Banach's fixed
  point theorem in \((E, d_{H})\) to operators satisfying the conditions of
  Lemma~\ref{lem:completeness_and_properties_positive_fcts}
  delivers the existence of a unique
  positive eigenfunction for \(A\). This is a variant of the Krein-Rutman theorem. The
  formulation we propose here is convenient because of its natural extension to
  random dynamical systems.

\end{remark}

\section{A Random Krein-Rutman Theorem}

In this section we reformulate the results of \cite{ArnoldGundlachDemetrius1994,
Hennion1997}, which refer to the case of positive random matrices, for positive
operators on Banach spaces.

An \emph{invertible metric discrete dynamical system} (IDS) \( (\Omega, \mF,
\PP, \vt)\) is a probability space \( ( \Omega, \mF, \PP)\) together with a
measurable map \(\vt \colon \ZZ \times \Omega \to \Omega\) such that \(\vt(z {+}
z^{\prime}, \cdot) = \vt(z, \vt(z^{\prime}, \cdot))\) and \(\vt(0, \omega) =
\omega\) for all \(\omega \in \Omega\), and such that \(\PP\) is invariant under
\(\vt(z, \cdot)\) for \(z \in \ZZ\). For brevity we write \(\vt^{z}( \cdot)\)
for the map \(\vt(z, \cdot)\). A set \( \widetilde{\Omega} \subseteq \Omega\) is
said to be \emph{invariant} for \(\vt\) if \(\vt^{z} \widetilde{\Omega} =
\widetilde{\Omega}\), for all \(z \in \ZZ\). An IDS is said to be
\emph{ergodic} if any invariant set \( \widetilde{\Omega}\) satisfies \(\PP(
\widetilde{\Omega}) \in \{ 0, 1\}\) (cf. \cite[Appendix
A]{Arnold1998RandomDynamicalSystems}).

Consider \(X, E\) as in the previous section and, for a given IDS, a random
variable \(A  \colon \Omega \to \mL(X)\). This
generates a measurable, linear, discrete random dynamical system (RDS) (see
\cite[Definition 1.1.1]{Arnold1998RandomDynamicalSystems}) \(\varphi\) on \(X\)
by defining:
\begin{equation}\label{eqn:discrete-dynamical-system-structure}
 \varphi_{n}(\omega) x = A( \vt^{n} \omega) \cdots A(\omega) x, \qquad n \in
\NN_{0}.
\end{equation}
If \(A( \omega )\) is in addition positive for every \(\omega \in \Omega\) (we
then simply say that \(A\) is positive), we can interpret \(\varphi\) as an RDS
on \(E\) via the projective action:
\[ \varphi_{n}^{\pi} (\omega) x = A^{\pi}( \vt^{n} \omega) \circ \cdots \circ A^{\pi}(\omega) x,
\qquad n \in \NN_{0}. \] 
Before we move on, let us recall the definition of invariant measures for
random dynamical systems, cf.  \cite[Section 1.4]{Arnold1998RandomDynamicalSystems}.
\begin{definition}\label{def:invariant-measure}
In the same setting as above, we say that a measure \(\mu\) on \(\Omega \times
E\) is invariant for \(\varphi^{\pi}\) if:
\begin{enumerate}[(i)]
\item The marginal \(\mu_{\Omega}\) of \(\mu\) on \(\Omega\) satisfies
\[ \mu_{\Omega} = \PP.\] 
\item The measure \(\mu\) is \(\Theta_{n}-\)invariant, where
\(\Theta_{n}\) is the skew-product
\[ \Theta_{n}(\omega, x) =(\vt^{n} \omega, \varphi_{n}^{\pi}(\omega)x). \] 
\end{enumerate}
\end{definition}
\begin{remark}\label{rem:on-invariant-measures}
In most cases an invariant measure \(\mu\) for a random dynamical system
\(\varphi\) admits a factorization of the form
\[ \mu(A \times B) = \int_{A \times B} \mu_{\omega}(\ud x) \PP( \ud \omega),\]
where \(A \subseteq \Omega\) and \(B \subseteq X\) are measurable sets, and
\(\omega \mapsto \mu_{\omega}(C)\) is a measurable function for every
measurable \(C \subseteq X\). We then identify the measure \(\mu\) with its
factor \(\mu_{\omega}\). In the setting of this article we will only deal with
random Dirac measures, of the form
\[ \mu_{\omega}(\ud x) = \delta_{x_{0}(\omega)},\] 
for a measurable map \(x_{0} \colon \Omega \to X\).
\end{remark}

\begin{assumption}\label{assu:random_compactness}
  Assume we are given \(X, K, U , E\) as in the previous
  section and that \( (E, d_{H})\) is a \emph{complete} metric
  space. Assume in addition that there exists an \emph{ergodic} IDS \(( \Omega,
  \mF, \PP, \vt)\).
  Let \(\varphi_{n}\) be a RDS defined via a random \emph{positive}
  operator \(A\) as above, such that:
  \[ \PP \Big( A \in \mL_{cp} (X) \Big) > 0.  \] 
\end{assumption}

In this setting the following is a random version of the Krein-Rutman theorem.
\begin{theorem}\label{thm:random_krein_rutman}

  Under Assumption \ref{assu:random_compactness}
  there exists a \(\vt-\)invariant set \( \widetilde{\Omega} \subseteq \Omega\)
  of full \(\PP-\)measure and a random variable \(u  \colon
  \Omega \to E\) such that:

  \begin{enumerate}[(i)]
    \item For all \(\omega \in \widetilde{\Omega}\) and \(f,g \in E\):
      \[ 
	\limsup_{n \to \infty} \bigg[ \frac{1}{n} \sup_{f, g \in E} \Big( \log{ d_{H} (
	    \varphi_{n}^{\pi} ( \omega) f , \varphi_{n}^{\pi} ( \omega) g)}
	\Big) \bigg] \leq \EE \log\big( \tau(A) \big) < 0. 
      \] 

    \item \(u\) is measurable w.r.t.\ to the \(\sigma-\)field 
      \( \mF^{{-}} = \sigma( (A(\vt^{{-} n} \cdot ) )_{n \in \NN})\) and:
      \[ \varphi_{n}^{\pi}( \omega) u(\omega) = u ( \vt^{n} \omega). \] 

    \item For all \(\omega \in \widetilde{\Omega}\):
      \[ \limsup_{n \to \infty} \bigg[ \frac{1}{n} \sup_{f \in E} \Big( \log
	    d_{H}( \varphi_{n}^{\pi} (
	\vt^{{-} n} \omega)  f, u( \omega) ) \Big) \bigg] \leq \EE \log \big( \tau(A)
      \big) < 0 \]
      as well as:
      \[  \limsup_{n \to \infty} \bigg[\frac{1}{n} \sup_{f \in E} \Big( \log
	    d_{H}( \varphi_{n}^{\pi} (
      \omega) f, u( \vt^{n} \omega) ) \Big)\bigg]\leq \EE \log \big(\tau(A) \big) < 0. \]

    \item The measure \(\delta_{u(\omega)}\) is the unique
      invariant measure for the RDS \( \varphi^{\pi}\) on \(E\). 

  \end{enumerate}

\end{theorem}

\begin{notation}

  We refer to the first property as \emph{asymptotic synchronization} and to
  the third property as \emph{one force, one solution} principle. 

\end{notation}

\begin{remark}\label{rem:continuous-time-random-krein-rutman}
Theorem~\ref{thm:random_krein_rutman} can be stated also in continuous
time. Suppose that \(\vt \colon \RR \times \Omega \to \Omega \) generates an invertible,
measure-preserving and ergodic dynamical system over \((\Omega, \mF, \PP)\) and
\[ \varphi \colon \RR_{+} \times \Omega \times X \to X\] 
defines a linear (i.e. \(\varphi_{t}(\omega) \in \mL(X), \ \forall t \geq 0,
\omega \in \Omega\)) random dynamical system (see \cite[Definition
1.1.1]{Arnold1998RandomDynamicalSystems}). Assume in addition that
\[ \varphi_{t}(\omega) \  \text{ is positive } \  \forall t \geq 0, \  \omega \in \Omega,
\qquad \PP ( \varphi_{1} \in \mL_{\mathrm{cp}}(X) ) >0. \] 
Then there exists a \(\vt-\)invariant set \(\widetilde{\Omega}\) of full
\(\PP-\)measure, such that for all \(\omega \in
\widetilde{\Omega}\):
\begin{align*}
\limsup_{t \to \infty} \bigg[ \frac{1}{t} \sup_{f, g \in E} \Big( \log{ \big( d_{H}(\varphi^{\pi}_{t}(\omega) f,
\varphi^{\pi}_{t}(\omega) g) \big)} \Big) \bigg] \leqslant \EE \log{\big( \tau
(\varphi_{1})\big) < 0}.
\end{align*}
And similarly one can adapt the properties at the points \((ii)-(iv)\) of
Theorem~\ref{thm:random_krein_rutman}. This extension follows directly from the
discrete case, observing that for \(n = \lfloor t \rfloor\), since
\(\tau ( \cdot) \leqslant 1\):
\begin{align*}
\log{(d_{H} \big(\varphi_{t}^{\pi}(\omega) f, \varphi_{t}^{\pi}(\omega) g)
\big)} & \leq \log{\big( \tau (\varphi_{t - n}(\vt^{-n}\omega))
d_{H}(\varphi^{\pi}_{n}(\omega) f, \varphi_{n}^{\pi}(\omega)g) \big)}\\
& \leq \log{\big( d_{H}(\varphi^{\pi}_{n}(\omega) f,
\varphi_{n}^{\pi}(\omega)g) \big)}.
\end{align*}
Then one can apply Theorem~\ref{thm:random_krein_rutman} since any
discrete random dynamical system has the
form~\eqref{eqn:discrete-dynamical-system-structure}, with \(A(\omega) =
\varphi_{1}(\omega)\).
\end{remark}
The proof of Theorem~\ref{thm:random_krein_rutman} will rely on the following
lemma.

\begin{lemma}\label{lem:1f1s}

  There exists a \(\vt-\)invariant set \(\widetilde{\Omega} \subseteq
  \Omega\) of full \(\PP-\)measure and an \(\mF^{{-}}-\)adapted random variable
  \(u: \Omega \to E\) such that:
  \[ \varphi_{n}^{\pi} ( \omega) u(\omega) = u (\vt^{n} \omega), \qquad \forall
  \omega \in \widetilde{\Omega}, n \in \NN. \]
  Moreover for all \(\omega \in \widetilde{\Omega}\):
  \[ 
    \limsup_{n \to \infty} \bigg[\frac{1}{n} \sup_{f \in E}\Big( \log  d_{H}(
	  \varphi_{n}^{\pi}( \vt^{{-} n}
    \omega) f, u (\omega)) \Big) \bigg]\leq \EE \log{\big( \tau(A) \big)}
    <0.
  \] 
\end{lemma}

\begin{proof}

  We start by observing (as in \cite[Proof of Lemma 3.3]{Hennion1997}) that the
  sequence of sets \(F_{n}(\omega) = \varphi_{n}^{\pi}( \vt^{{-} n} \omega )(
    E)\) is decreasing, i.e.  \(F_{n {+} 1 } \subseteq F_{n}\). Let us write
  \(F(\omega) = \bigcap_{n \geq 1} F_{n}(\omega)\). Hence by
  Theorem~\ref{thm:Birkhoff_positive_operators}:
  \[ 
    \Delta (F) \leq  \lim_{ n \to \infty} \Delta (F_{n}) = \lim_{n \to \infty} 4
    \ \mathrm{arctanh} \ ( \tau ( \varphi_{n} (\vt^{{-} n} \omega))). 
  \] 
  Now, by the ergodic theorem and Assumption~\ref{assu:random_compactness} there exists a \(\vt-\)invariant set \( \widetilde{\Omega}\)
  of full \(\PP-\)measure such that for all \(\omega \in \widetilde{\Omega}\):
  \begin{equation}\label{eqn:proof_ofos_support} 
    \limsup_{n \to \infty} \frac{1}{n} \log \Big( \tau ( \varphi_{n}( \vt^{{-} n}
    \omega)) \Big) \leq \lim_{n \to \infty} \frac{1}{n} \sum_{i = 0}^{n} \log{ \tau(A
    (\vt^{{-} i} \omega))} = \EE  \log{\big( \tau(A)\big) }  <0.   
  \end{equation} 
  In particular \(\lim_{n \to \infty} \tau \Big( \varphi_{n}(\theta^{-n}
\omega) \Big) = 0 \), and since \(\mathrm{arctanh}(0) = 0\) we have that \(\Delta(F) = 0\). By completeness of \(E\) it follows that
  \(F\) is a singleton. Let us write \(F(\omega) = \{u(\omega)\}\) and extend
  \(u\) trivially outside of \( \widetilde{\Omega}\): it is clear that
  \(u\) is adapted to \(\mF^{{-}}\). Since for \(k \in \NN\) and \(n \geq k\)
  \[ 
    \varphi_{n} (\vt^{{-} n} \vt^{k} \omega ) = \varphi_{k}(\omega) \circ \varphi_{n {-} k}( \vt^{{-} (n {-} k)}\omega), 
  \] 
  passing to the limit with \(n \to \infty\) we have: \(u( \vt^{k} \omega) =
    \varphi_{k}^{\pi} (\omega) u( \omega) \).

  Finally, as in the former result, a Taylor expansion guarantees that there
  exists a constant \(c(\omega)>0\) such that:
  \[ 
    \Delta ( \varphi_{n}^{\pi} ( \vt^{{-} n} \omega) (E)) =  4 \
    \mathrm{arctanh} \ ( \tau ( \varphi_{n} (\vt^{{-} n} \omega))) \leq  4 (1 +
    c(\omega))\tau ( \varphi_{n} ( \vt^{{-} n} \omega)).  
  \] 
  This estimate, combined with the fact that
  \[ 
    \sup_{f \in E} d_{H} ( \varphi_{n}^{\pi} ( \vt^{{-} n} \omega) f , u(
    \omega)) = \sup_{f \in E} d_{H} (\varphi_{n}^{\pi} ( \vt^{ {-} n}
\omega) f,
    \varphi_{n}^{\pi} ( \vt^{{-} n} \omega) u( \vt^{{-} n}\omega)) \leq \Delta (
    \varphi_{n}^{\pi} ( \vt^{{-} n} \omega) (E)) 
  \] 
  and \eqref{eqn:proof_ofos_support} provides the required convergence result.

\end{proof}

\begin{proof}[Proof of Theorem~\ref{thm:random_krein_rutman}]

  As for the first property, compute:
  \begin{align*}
    d_{H}( \varphi_{n}^{\pi} (\omega) f,
    \varphi_{n}^{\pi}(\omega) g) \leq & \tau(A_{n} (\omega))
    d_{H}( \varphi_{n {-} 1}^{\pi}( \omega ) f ,
    \varphi_{n {-} 1}^{\pi} (\omega) g) \\
    \leq & \prod_{i = 0}^{n} \tau(A( \vt^{i} \omega)) d_{H}(f, g).
  \end{align*}
  Then, applying the logarithm and Birkhoff's ergodic theorem we find:
  \[ \limsup_{n \to \infty} \frac{1}{n} \log{ (\tau (\varphi_{n}(\omega)) )}
    \leq \EE \log{( \tau(A))} < 0.  \] 
  If \(\EE \log{ \big( \tau(A)\big)} = {-} \infty\) we can instead follow the
  previous computation with \(\tau(A ( \vt^{i} \omega))\) replaced by
  \(\tau(A(\vt^{i} \omega)) \vee e^{{-} M}\) and eventually pass to the limit
  \(M \to \infty\).
  To obtain the result uniformly over \(f,g\) first observe that via
  Theorem~\ref{thm:Birkhoff_positive_operators}:
  \[ 
    \sup_{f, g \in E} \Big( \log{ d_{H} (
      \varphi_{n}^{\pi} ( \omega) f , \varphi_{n}^{\pi} ( \omega) g)}
    \Big) = \log \Big( \Delta ( \varphi_{n}^{\pi}(\omega) (E) )\Big) = \log \Big(
    4 \ \mathrm{arctanh} (\tau( \varphi_{n} (\omega)))\Big),
  \]
  and by a Taylor approximation, since \(\lim_{n \to \infty}
  \tau(\varphi_{n}(\omega)) = 0\), there exists a constant \(c(\omega)>0\) such
  that
  \begin{align*}
    \limsup_{n \to \infty} \frac{1}{n} \log \Big( 4 \ \mathrm{arctanh} (\tau(
    \varphi_{n} (\omega))) \Big) & \leq \limsup_{n \to \infty} \frac{1}{n} \log
    \Big( (1 + c(\omega)) \tau( \varphi_{n} (\omega)) \Big)  \\
    & = \limsup_{n \to \infty} \frac{1}{n} \log
    \Big( \tau( \varphi_{n} (\omega)) \Big)  \leq  \EE \log{ \tau(A)}. 
  \end{align*}

Point \((ii)\) as well as the first property of \((iii)\) follow from
  Lemma \ref{lem:1f1s}.
As for the second property of \((iii)\) we observe that
\begin{align*}
\sup_{f \in E} \Big( \log d_{H}( \varphi_{n}^{\pi} ( \omega) f, u( \vt^{n}
\omega) ) \Big) & =\sup_{f \in E} \Big( \log d_{H}( \varphi_{n}^{\pi} ( \omega)
f, \varphi_{n}^{\pi} (\omega) u(\omega) ) \Big)\\ 
& \leq \sup_{f,g \in E} \Big( \log d_{H}( \varphi_{n}^{\pi} ( \omega)
f, \varphi_{n}^{\pi} (\omega)g ) \Big),
\end{align*}
so that the estimate is now a consequence of point \((i)\). As for
\((iv)\), we have that for any two measurable \(A \subseteq \Omega, B \subseteq
E\):
\begin{align*}
\int_{\Omega \times E} 1_{A}(\vt^{n} \omega) 1_{B}(\varphi_{n}^{\pi}(\omega) f)
\delta_{u(\omega)}(\ud f) \PP(\ud \omega) & = \int_{\Omega \times E}
1_{A}(\vt^{n} \omega) 1_{B}(u(\vt^{n} \omega)) \PP(\ud \omega) \\
& = \int_{\Omega \times E} 1_{A}(\omega) 1_{B}(u(\omega)) \PP (\ud \omega),
\end{align*}
which implies that \(\delta_{u(\omega)}\) is invariant (see
Definition~\ref{def:invariant-measure}). Finally, to see that
\(\delta_{u(\omega)}\) is the unique invariant measure, let \(\mu\) be any
invariant measure. Then
\begin{align*}
 \int_{\Omega \times E} \min \{1, d_{H}(f, u( \omega)) \} \mu( \ud \omega, \ud
f) & = \lim_{n \to \infty} \int_{\Omega \times E}\min \{ 1, d_{H}(
\varphi_{n}^{\pi}(\omega) f, u(\vt^{n} \omega))\} \mu(\ud \omega, \ud
f) \\
& \leq \lim_{n \to \infty}\int_{\Omega} \sup_{f \in E}
\min \{1, d_{H}(\varphi^{\pi}_{n}(\omega) f , u(\vt^{n} \omega)) \}\PP(\ud
\omega)\\
& \leq 0,
\end{align*}
where in the last line we used dominated convergence and the results of point
\((iii)\). In particular, we have found that
\[ \mu ( \{ (\omega, f) \in \Omega \times E  \ \colon \ f \neq u(\omega)\}) =0, \] 
implying that \(\mu(\ud \omega, \ud f) = \delta_{u (\omega)} (\ud f)
\PP(\ud \omega)\).
 Note that the invariant sets in all points can be
  chosen equal to the same \( \widetilde{\Omega}\) up to taking
  intersections of invariant sets, which are still invariant.

\end{proof}

\section{Synchronization for linear SPDEs}\label{sec:syncrhonization-spdes}

In this section we discuss how to apply the previous results to stochastic PDEs.
Concrete examples will be covered in the next section. For clarity, nonetheless,
the reader should keep in mind that we want to study ergodic properties of
solutions to Equation \eqref{eqn:intro_kpz}. Since the associated heat equation
with multiplicative noise \eqref{eqn:intro_she} is linear and the solution map is
expected to be strictly positive (because the defining differential operator is
parabolic), we may assume that the solution map generates
a continuous, linear, strictly positive random dynamical system \(\varphi\).

\begin{definition}\label{def:discrete-continuous-dyn-system}
   
  A \emph{continuous} RDS over a \emph{discrete} IDS \((\Omega , \mF,
  \PP, \vt)\) and on a measure space \( (X, \mB)\) is a map 
  \[\varphi  \colon \RR_{+} \times \Omega \times X \to X\]
  such that the following two properties hold:
  \begin{enumerate}[(i)]
    \item \emph{Measurability:} \(\varphi\) is \( \mB( \RR_{+}) \otimes \mF
      \otimes \mB\)-measurable.

    \item \emph{Cocycle property:} \(\varphi(0, \omega)= \mathrm{Id}_{X},\) for
      all \(\omega \in \Omega\) and:
      \[ \varphi(t {+} n , \omega)= \varphi( t, \vt^{n} \omega) \circ
      \varphi(n, \omega), \qquad \forall t \in \RR_{+}, n \in \NN_{0}, \omega
    \in \Omega.\] 

  \end{enumerate}

\end{definition}

We then formulate the following assumptions, under which our main result will
hold.

\begin{assumption}\label{assu:properties_moments_of_solution_map}

  Let \(d \in \NN\) and \(\beta> 0\). Let \( (\Omega_{\mathrm{kpz}}, \mF, \PP,
  \vt)\) be a discrete ergodic IDS, over which is defined a continuous RDS
  \(\varphi\):
  \[
    \varphi : \RR_{+} \times \Omega_{\mathrm{kpz}} \to \mL(C(\TT^{d})).
  \]
  There exists a \(\vt-\)invariant set \( \widetilde{\Omega} \subseteq
  \Omega_{\mathrm{kpz}}\) of full \(\PP-\)measure such that the following properties
  are satisfied for all \(\omega \in \widetilde{\Omega}\) and any \(T> S >
  0\):

  \begin{enumerate}[(i)]
    \item There exists a kernel \(K: \Omega_{\mathrm{kpz}} \to
	C_{\mathrm{loc}}((0, \infty); C(\TT^{d}
      \times \TT^{d}))\) such that for all \(S \leq t \leq T\):
      \[ 
	\varphi_{t}(\omega) f (x)   = \int_{\TT^{d}} K( \omega, t, x,y) f(y) \ud
	y, \qquad \forall f \in C(\TT^{d}), x \in \TT^{d}. 
      \] 

    \item There exist \(0 < \gamma (\omega, S, T) \leq \delta (\omega, S,
      T)\) such that: 
      \[ \gamma (\omega, S, T) \leq  K(\omega,t, x,y) \leq \delta (\omega,
      S, T ), \qquad \forall x,y \in \TT^{d}, \ S \leq t \leq T, \] 
      which implies that \(\PP \big( \varphi_{t} \in \mL_{\mathrm{cp}} (C
      (\TT^{d})), \forall t \in (0, \infty) \big) = 1.\)

    \item There exists a constant \(C(\beta, \omega, S, T)\) such that:
      \[ \| \varphi_{t} f \|_{\beta} \leq C(\beta, \omega, S, T) \| f
      \|_{\infty}, \qquad \forall f \in C(\TT^{d}), \ S \leq t \leq T. \] 

    \item Consider \( (E, d_{H})\) as in
      Lemma~\ref{lem:completeness_and_properties_positive_fcts}. The following moment estimates are satisfied for any \(f \in E\): 
      \[ 
	\EE \log {\big( C ( \beta, S, T) \big)} + \EE \sup_{ S \leq t \leq T}
	d_{H}( \varphi_{t}^{\pi}f, f) < {+} \infty, 
      \]
      where \(\varphi_{t}^{\pi}\) is defined to be the identity outside of
      \( \widetilde{\Omega}\).
  \end{enumerate}

\end{assumption}

The first two assumptions allow us to use the results from the
previous section. The last two will guarantee convergence in appropriate
H\"older spaces. In view of the motivating example and in the setting of the previous
assumption, we say that for \(z \in \ZZ\) and \(h_{0} \in C(\TT^{d})\) the map
\[ 
  [z, {+} \infty) \times \TT^{d} \ni (t, x) \mapsto
  h^{z}(\omega , t, x), \qquad h^{z}(\omega, z, x) = h_0(x) 
\] 
solves Equation \eqref{eqn:intro_kpz} if \(h^{z}(\omega, t) = \log{ \big(
\varphi_{t}( \vt^{z} \omega) \exp(h_0) \big)}\) for \(\varphi_{t}\) as in the
previous assumption.

\begin{theorem}\label{thm:synchronization_for_kpz}

  Under Assumption \ref{assu:properties_moments_of_solution_map},
  for \(i = 1,2\), \(h_{0}^{i} \in C(\TT^{d})\) and \(n \in \NN\),
  let \(h_i(t) \in C (\TT^{d})\) be the random solution to Equation
  \eqref{eqn:intro_kpz}
  started at time $0$ with initial data \(h_0^{i}\) and evaluated at time $t \geq 0$.
  Similarly, let \(h^{{-} n}_{i}(t) \in C (\TT^{d})\) be the solution started in
  \({-} n\) with initial data \(h_{0}^{i}\) and evaluated at time \(t \geq {-} n\).
  There exists an invariant set \( \widetilde{\Omega} \subseteq
  \Omega_{\mathrm{kpz}}\) of full \(\PP-\)measure such that for any
  \(0 < \alpha< \beta\), for any \(T>0\) and any
  \(\omega \in \widetilde{\Omega}\):

  \begin{enumerate}[(i)]
    \item There exists a map
      \(c(h^{1}_{0}, h^{2}_{0}) \colon \Omega_{\mathrm{kpz}} \times \RR_{+}
      \to \RR\) such that: 
      \[ 
	\limsup_{n \to \infty} \bigg[ \frac{1}{n} \log \Big( \sup_{t \in [n, n {+} T]} \|
    	  h_{1} ( \omega, t) - h_{2} (\omega, t) {-} c(\omega, t,
      h^{1}_{0}, h^{2}_{0}) \|_{C^{\alpha}(\TT^{d})} \Big) \bigg] \leq
\Big( 1 - \frac{\alpha}{ \beta} \Big) \EE \log
	\big( \tau (\varphi_{1})\big) < 0, 
      \] 
      as well as:
      \[ 
	\limsup_{n \to \infty}\bigg[ \frac{1}{n} \log{ \Big( \sup_{t \in [n , n {+} T]}
	[h_{i}( \omega, t)]_{\beta}} \Big)\bigg]\leq 0. 
      \] 
      And uniformly over \(h_{0}^{i}\):
      \[ 
	\limsup_{n \to \infty} \bigg[ \frac{1}{n} \log \Big( 
	    \sup_{\substack{ h_{0}^{i} \in C(\TT^{d}), \\ t \in [n, n {+} T]}} \|
      	    h_{1} ( \omega, t) - h_{2} (\omega, t) {-} c(\omega, t,
        h^{1}_{0}, h^{2}_{0}) \|_{\infty} \Big) \bigg] \leq \EE \log
	\big( \tau (\varphi_{1})\big) < 0, 
      \] 
    \item There exists a random function \(h_{\infty}  \colon
      \Omega_{\mathrm{kpz}} \to C_{\mathrm{loc}}(({-} \infty, \infty);
      C^{\alpha}(\TT^{d}))\) and a
      sequence of maps \(c^{{-} n}(h^{1}_{0}) \colon \Omega_{\mathrm{kpz}} \times
      \RR_{+} \to \RR\) for which:
      \[ 
	\limsup_{n \to \infty} \bigg[ \frac{1}{n} \log \Big( \!\! \! \! \! \!
	    \sup_{\substack{ h_{0}^{1} \in C(\TT^{d}), \\ t \in [({-} T)\vee
	    ({-} n), T]} } \! \! \! \! \! \! \| h^{{-} n}_{1} (\omega, t) -
	h_{\infty} (\omega, t) - c^{{-} n}(\omega, t, h^{1}_{0})
    \|_{C^{\alpha}(\TT^{d})} \Big)\bigg] \leq \Big( 1 -
\frac{\alpha}{\beta} \Big)\EE \log
    \big(\tau(\varphi_{1})\big) < 0. 
      \]

  \end{enumerate}

\end{theorem}

Passing to the gradient one can omit all constants and find the following
for Burgers' Equation.

\begin{corollary}

  In the same setting as before, it immediately follows that also:
  \begin{align*}
    \limsup_{n \to \infty} \bigg[ \frac{1}{n} \log  \Big( \sup_{t \in [n, n {+} T]} \|
        \nabla_{x} h_{1} ( \omega, t) - \nabla_{x} h_{2} (\omega, t)
    \|_{C^{\alpha {-} 1}(\TT^{d})} \Big) \bigg] & \leq  \Big( 1 -
\frac{\alpha}{\beta} \Big) \EE \log{
    \big( \tau (\varphi_{1})\big)} < 0,\\ 
    \limsup_{n \to \infty} \bigg[ \frac{1}{n}  \log \Big( \! \! \! \! \sup_{\substack{ h^{1}_{0} \in
	C(\TT^{d}), \\ t \in [({-} T)\vee ({-} n), T]}} \! \! \! \! \| \nabla_{x} h^{{-} n}_{1} (\omega, t)
      - \nabla_{x} h_{\infty} (\omega, t)\|_{C^{\alpha {-} 1}(\TT^{d})} \Big)
    \bigg] & \leq  \Big( 1 - \frac{\alpha}{\beta} \Big)  \EE \log \big( \tau(\varphi_{1})\big) < 0, 
  \end{align*}
  where the space \(C^{\alpha {-} 1} (\TT^{d})\) is understood as the
  Besov space \(B^{\alpha {-} 1}_{\infty, \infty}(\TT^{d})\) for \(\alpha \in
  (0,1)\).

\end{corollary}

\begin{proof}[Proof of Theorem \ref{thm:synchronization_for_kpz}]

  Consider \(\omega \in \widetilde{\Omega}\) and \(T >0\) fixed for the entire
  proof. 
  
  \textit{Step 1.} Define:
  \[u_{0}^{i} = \exp (h_{0}^{i}) / \|\exp ( h_0^{i} ) \|_{L^{1}} \in  E, \] 
  so that \(h_{i} (\omega, t) =
    \log{ \big( \varphi_{t}^{\pi} (\omega)  u_{0}^{i} \big) } {+} c_{i}(\omega,
  t)\), where \(c_{i}(\omega, t) \in \RR\) is the normalization constant:
  \[ 
    c_{i}(\omega, t) = \log \bigg( \int_{\TT^{d}} (\varphi_{t} (\omega)
    u_{0}^{i})  (x) \ud x \bigg) + \log{ \bigg( \int_{\TT^{d}} \exp
    (h_0^{i})(x) \ud x \bigg)}. 
  \]
  Let us write \(c(\omega, t, h^{1}_{0}, h^{2}_{0}) = c_{1}(\omega, t) {-}
  c_{2}(\omega, t)\). Similarly, for \({-} n \leq t \leq 0\) one has:
  \[
    h^{{-} n}_{i} (\omega, t) = \log{ \big( \varphi_{n {+} t}^{\pi}( \vt^{{-} n} \omega)
    u_{0}^{i}\big)} {+} c_{i}^{{-} n}(\omega, t) = h_{i}( \vt^{{-}n}
    \omega, n {+} t),
  \] 
  where \(c_{i}^{{-} n}( \omega, t) = c_{i}( \vt^{{-} n} \omega, n {+} t)\).
  Also, write \(c^{{-} n}(\omega, t, h^{1}_{0}, h^{2}_{0}) = c_{1}^{{-} n}(\omega, t) {-}
  c_{2}^{{-} n}( \omega, t)\). As a first step, we prove the following simpler
- since it considers convergence in $ L^{\infty} $ instead of $ C^{\alpha} $ -  version of
  the required result:
  \begin{equation}\label{eqn:support_time_inft_norm}
    \begin{aligned}
      \limsup_{n \to \infty} \bigg[ \frac{1}{n} \log{ \Big( \sup_{\substack{
	    h_{0}^{i} \in C(\TT^{d}), \\ t \in [n, n {+} T]}} \|
          h_1(\omega, t) {-} h_{2}(\omega, t) {-} c(\omega, t, h^{1}_{0},
      h^{2}_{0}) \|_{\infty} \Big) } \bigg] &\\
      \leq \EE \log \big( & \tau (\varphi_{1})\big),\\
      \limsup_{n \to \infty}  \bigg[ \frac{1}{n} \log{ \Big( \! \! \! \! \sup_{\substack{
	      h_{0}^{1} \in C(\TT^{d}) \\ t \in [({-} T) \vee ({-}
	  n), T]}} \! \! \! \! \| h_1^{{-} n}(\omega, t) {-} h_{\infty}(\omega, t) {-} c^{{-}
    n}_{1}(\omega, t) \|_{\infty}} \Big) \bigg]&\\
    \leq \EE \log \big( & \tau( \varphi_{1})\big).
    \end{aligned}
  \end{equation}
 We observe that in view of Assumption
  \ref{assu:properties_moments_of_solution_map}, we can apply Theorem
  \ref{thm:random_krein_rutman} in the setting of
  Lemma~\ref{lem:completeness_and_properties_positive_fcts} with \(A(\omega) = \varphi_{1}(\omega)\) to see
  that there exists a \( \overline{u}_{\infty} = \exp( \overline{h}_{\infty}) \colon
  \Omega_{\mathrm{kpz}} \to C(\TT^{d})\) such that \( \varphi^{\pi}_{n}
(\omega) \overline{u}_{\infty} (\omega) = \overline{u}_{\infty}(\vt^{n} \omega)
\). In particular, we define for any $ t \in \RR $: \[ h_{\infty}(\omega, t) = \log(
\varphi_{t+ n}^{\pi}(\vt^{- n}\omega) \overline{u}_{\infty}(\vt^{- n}\omega ) )
= \log{ u_{\infty}(\omega, t)},\] for
any $ n $ such that $ t+n > 0 $ (note that the definition does not depend on
the choice of such $ n $). 
  With this definition we proceed to prove~\eqref{eqn:support_time_inft_norm}. We start by eliminating the time supremum, since in view of
  Inequality \eqref{eqn:log_hilbert_distance_equivalence}:
  \begin{align*}
    \limsup_{n \to \infty}  \bigg[ \frac{1}{n} & \log{ \Big( \sup_{\substack{
	  h_{0}^{i} \in C(\TT^{d}) \\ t \in [n, n {+} T]}} \|
        h_1(\omega, t) {-} h_{2}(\omega, t) {-} c(\omega, t, h^{1}_{0},
    h^{2}_{0}) \|_{\infty}\Big)}  \bigg] \\
    & \leq \limsup _{n \to \infty} \bigg[ \frac{1}{n} \log{ \Big(
	  \sup_{\substack{ h_{0}^{i} \in C(\TT^{d}) \\ t \in [n, n {+} T]}}
	  d_{H}( \varphi_{t}^{\pi}(\omega) u^{1}_{0}, \varphi_{t}^{\pi} (\omega) 
    u_{0}^{2})\Big) } \bigg] \\
    & \leq \limsup_{n \to \infty} \bigg[ \frac{1}{n} \log{ \Big( \sup_{h_{0}^{i} \in C
	  (\TT^{d})} d_{H} ( \varphi_{n}^{\pi}(\omega) u_0^1,
	    \varphi_{n}^{\pi} (\omega) u_0^2) \Big)} \bigg]
  \end{align*}
where we used the definition of the contraction constant \(\tau( \cdot)\)
together with the fact that \(\tau(\cdot) \leq 1\) (cf.
Theorem~\ref{thm:Birkhoff_positive_operators}) to obtain
\begin{align*}
d_{H}( \varphi_{t}^{\pi}(\omega) u_0^1, \varphi_{t}^{\pi} (\omega) u_0^2)
& = d_{H}( \varphi_{t-n}^{\pi}(\vt^{n} \omega) \varphi_{n}^{\pi}(\omega) u_0^1,
\varphi_{t-n}^{\pi}(\vt^{n} \omega) \varphi_{n}^{\pi} (\omega) u_0^2) \\
& \leq \tau( \varphi_{t-n}^{\pi}(\vt^{n} \omega))
d_{H}(\varphi_{n}^{\pi}(\omega) u_{0}^{1}, \varphi_{n}^{\pi}(\omega)
u_{0}^{2}) \\
& \leq d_{H}(\varphi_{n}^{\pi}(\omega) u_{0}^{1}, \varphi_{n}^{\pi}(\omega),
u_{0}^{2})
\end{align*}
so that one can estimate:
  \begin{align*}
    \sup_{t \in [n, n {+} T]}d_{H}( \varphi_{t}^{\pi}(\omega) u_0^1,
    \varphi_{t}^{\pi} (\omega) u_0^2) \leq  d_{H}( \varphi_{n}^{\pi}(\omega)
    u_0^1, \varphi_{n}^{\pi} (\omega) u_0^2).
  \end{align*}
Similarly, also for the backwards case: 
\begin{align*}
 \limsup_{n \to \infty} & \bigg[ \frac{1}{n} \log{ \Big( \! \! \! \!
\sup_{\substack{ h_{0}^{1} \in C(\TT^{d}) \\ t \in [({-} T) \vee ({-} n), T]}}
\! \! \! \! \| h_1^{{-} n}(\omega, t) {-} h_{\infty}(\omega, t) {-} c^{{-}
n}_{1}(\omega, t) \|_{\infty}} \Big) \bigg] \\
& \leqslant \limsup_{n \to \infty} \bigg[ \frac{1}{n} \log \bigg( \sup_{\substack{ h_{0}^{1} \in C(\TT^{d}) \\ t \in [({-} T) \vee ({-} n), T]}}
\! \! \! \!d_{H}( \varphi_{n + t}^{\pi}
(\vt^{- n} \omega) u_{0}^{1}, \varphi_{n + t}^{\pi} (\vt^{- n} \omega)
\overline{u}_{\infty}(\vt^{- n} \omega)) \bigg) \bigg]\\
& \leqslant \limsup_{n \to \infty} \bigg[ \frac{1}{n} \log \bigg( \sup_{ h_{0}^{1} \in C(\TT^{d})}
\! \! \! \!d_{H}( \varphi_{n - T}^{\pi}
(\vt^{- n} \omega) u_{0}^{1}, \varphi_{n - T}^{\pi} (\vt^{- n} \omega)
\overline{u}_{\infty}(\vt^{- n} \omega)) \bigg) \bigg].
\end{align*}
Now, again in view of Assumption \ref{assu:properties_moments_of_solution_map},
we can apply Theorem \ref{thm:random_krein_rutman} to obtain:
  \begin{align*}
    \limsup_{n \to \infty} & \bigg[ \frac{1}{n} \log {\Big( \sup_{u_{0}^{i} \in
	  E} d_{H}( \varphi_{n}^{\pi}(\omega) 
  u^{1}_{0}, \varphi_{n}^{\pi} ( \omega) u^{2}_{0} ) \Big)} \bigg] \leq \EE
    \log{ \big( \tau( \varphi_{1}) \big)}, \\
    \limsup_{n \to \infty} & \bigg[ \frac{1}{n} \log \Big( \sup_{u^{1}_{0} \in E}
	d_{H}( \varphi_{n}^{\pi} (\vt^{{-} n}\omega) u^{1}_{0},
    \overline{u}_{\infty}(\omega) ) \Big) \bigg] \leq \EE \log{ \big( \tau (
    \varphi_{1}) \big)}, 
  \end{align*}
  which via the previous calculation implies
  \eqref{eqn:support_time_inft_norm}. In particular, this also proves the bound
  uniformly over \(h_{0}^{i}\) at point \((i)\) of the theorem.

  \textit{ Step 2.} We pass to prove convergence in \(C^{\alpha}(\TT^{d})\) for
  \(0 < \alpha< \beta\). Hence consider
  \(\alpha < \beta\) fixed and define \(\theta \in (0,1)\) by \(\alpha = \beta \theta\). As
  convergence in \(C(\TT^{d})\) is already established, to prove
  convergence in \(C^{\alpha}(\TT^{d})\) one has to control the
  \(\alpha-\)seminorm \([ \cdot]_{\alpha}\) of \(h_{1}- h_{2}\). We treat the forwards and backwards
  in time cases differently, starting with the first case.
  Let us recall the bound
\begin{align*}
[f]_{\alpha} \leq C(\alpha, \beta) \| f \|_{\infty}^{1- \theta} [f]_{\beta}^{\theta},
\end{align*}
which is proven in Lemma~\ref{lem:interpolation-bound-general-regularity}.
With this bound one can estimate the H\"older seminorm via:
  \begin{equation}\label{eqn:interpolation_in_proof}
    \begin{aligned}
      & [h_{1}(\omega, t) {-} h_{2}(\omega, t) {-} c(\omega, t,
      h^{1}_{0}, h^{2}_{0})]_{\alpha} \leq C(\alpha, \beta) \| h_{1}( \omega, t) {-} h_{2}( \omega, t) {-}
	c(\omega, t, h^{1}_{0}, h^{2}_{0})
      \|_{\infty}^{1 {-} \theta} \cdot \\
      & \qquad \qquad \qquad \qquad \qquad \qquad \qquad \qquad \qquad \cdot
      \Big( [ \log{ (\varphi_{t}^{\pi} (\omega)
	u^{1}_{0} )} ]_{\beta} {+} [ \log{( \varphi_{t}^{\pi}(\omega)
      u^{2}_{0} )} ]_{\beta}\Big)^{\theta}. 
    \end{aligned}
  \end{equation}
Here we used that for the H\"older seminorm
\begin{align*}
[h_{1}(\omega, t) - h_{2}(\omega, t) - c(\omega, t, h^{2}_{0} ,
h^{2}_{0})]_{\beta} & = [\log{ (\varphi_{t}^{\pi} (\omega)
	u^{1}_{0} )}- \log{( \varphi_{t}^{\pi}(\omega) u^{2}_{0} )} ]_{\beta}\\
& \leqslant [ \log{ (\varphi_{t}^{\pi} (\omega)
	u^{1}_{0} )} ]_{\beta} {+} [ \log{( \varphi_{t}^{\pi}(\omega)
      u^{2}_{0} )} ]_{\beta},
\end{align*}
since the seminorm does not vary under translations by a constant.
  
Since we already proved that the first factor in the product vanishes exponentially fast, our aim will
be to prove that the second factor does not explode exponentially fast. This
amounts to proving the second bound at point \((i)\). To this end, fix
\(n \in \NN, T>0\) and \(t \in [n, n {+} T]\), and define \(\sigma\) by \(t = n {-} 1 {+}
\sigma\). We can use Lemma~\ref{lem:derivatives-of-the-logarithm} to bound
the last terms by: 
  \begin{align*}
    [h_{1}(\omega, t)]_{\beta} & = [ \log { ( \varphi_{t}^{\pi}( \omega)
    u^{1}_{0}) } ]_{\beta} \\
    & \leq \overline{C}(\beta) \left(  \frac{1+[ \varphi_{t}^{\pi} ( \omega) u^{1}_{0}
]_{\beta}}{m( \varphi_{t}^{\pi} ( \omega) u^{1}_{0})} \right)^{\lfloor \beta
\rfloor + 1} \\
		& \leqslant \overline{C}(\beta)  \left( \frac{1+[ \varphi_{\sigma}^{\pi}(
\vt^{n {-} 1} \omega) \circ \varphi_{n {-} 1}^{\pi}( \omega) u^1_0
]_{\beta}}{m( \varphi_{t}^{\pi} ( \omega) u^{1}_{0})}\right)^{\lfloor \beta
\rfloor +1} \\
		& \leq \overline{C}(\beta) \left( \frac{1+C( \beta, \vt^{n {-} 1}
\omega, 1, T {+} 1) \| \varphi_{n {-} 1}^{\pi}( \omega)  u^{1}_{0}
\|_{\infty}}{m( \varphi_{t}^{\pi} ( \omega) u^{1}_{0})}\right)^{\lfloor \beta
\rfloor +1}
  \end{align*}
  where \(m( \cdot)\) indicates the minimum of a function and \(
\overline{C}(\beta)\) is the deterministic constant of
Lemma~\ref{lem:derivatives-of-the-logarithm}. We can plug this
estimate into Equation~\eqref{eqn:interpolation_in_proof} to obtain for some
deterministic \( \widetilde{C}(\alpha, \beta)>0\):
  \begin{align*}
    \log [ h_{1}(\omega , t) {-} h_{2}(\omega, t) & {-} c(\omega, t,
      h^{1}_{0}, h^{2}_{0}) ]_{\alpha} \\
    \leq &  (1 {-} \theta) \log{\| h_{1}(\omega, t) {-}
    h_{2}(\omega, t) {-} c(\omega, t, h^{1}_{0}, h^{2}_{0})\|_{\infty}} \\ 
    & +  \theta (\lfloor \beta \rfloor +1) \log \Big( \sum_{i = 1,2} \frac{1+C(\beta, \vt^{n {-} 1}
	\omega, 1, T {+} 1)\|
    \varphi_{n {-} 1}^{\pi} (\omega) u^{i}_{0}
\|_{\infty}}{m(\varphi_{t}^{\pi}(\omega) u^{i}_{0})}  \Big) \\
& + \widetilde{C}(\alpha, \beta)\\
\leq & (1 {-} \theta) \log{ \| h_{1}(\omega, t) {-}
    h_{2}(\omega, t) {-} c(\omega, t, h^{1}_{0}, h^{2}_{0})\|_{\infty}} \\ 
    & +  \sum_{i=1,2} \theta ( \lfloor \beta \rfloor +1)\log \Big( 2\frac{(1+ C(\beta, \vt^{n {-} 1}
	\omega, 1, T {+} 1))\| \varphi_{n {-} 1}^{\pi} (\omega) u^{i}_{0}
\|_{\infty}}{m(\varphi_{t}^{\pi}(\omega) u^{i}_{0})}  \Big)\\
&  + \widetilde{C}(\alpha, \beta),
  \end{align*}
where in the last line we used that \(\log{ \max_{i} x_{i}} = \max_{i} \log{
x_{i}}\) and that \(\| \varphi_{n-1}^{\pi}(\omega) u_{0}^{i}  \|_{\infty}
\geqslant 1\),
since \(\varphi_{n-1}^{\pi}(\omega) u_{0}^{i} \in E\) and hence
\(\int_{\TT^{d}}\varphi_{n-1}^{\pi}(\omega) u_{0}^{i}(x) \ud x=1\).
  To conclude, in view of Equation~\eqref{eqn:support_time_inft_norm}, we have to prove
  that for \(i=1,2\):
  \begin{equation}\label{eqn:proof-synchro-holder-support-bound}
    \limsup_{n \to \infty} \bigg[ \frac{1}{n} \log{\bigg( \sup_{n \leq t \leq n {+} T } \Big(
    	 \frac{1+C(\beta, \vt^{n {-} 1} \omega, 1, T {+} 1)}{m(
	\varphi_{t}^{\pi}(\omega) u^{i}_{0})} \| \varphi_{n {-} 1}^{\pi}(\omega)
    u^{i}_{0} \|_{\infty} \Big) \bigg)} \bigg]\leq  0.  
  \end{equation}
  In particular, the latter inequality also implies the \(\beta\)-H\"older
  norm bound of \(h_{i}\) at point \((i)\) of the theorem.
  Now we observe that for \(n \in \NN, n \geq 1, \ T>0, t = n -1 + \sigma \in [n,
n + T]\) and for any \(f \in E\):
  \begin{align*}
      \log \Big( 1+ C &( \beta, \vt^{n {-} 1} \omega, 1, T {+} 1)\Big)  + \log{
\Big( \| \varphi_{n {-} 1}^{\pi} (\omega) f \|_{\infty}\Big)} {-} \log{ \Big( m
\big(\varphi_{t}^{\pi} ( \omega) f \big) \Big)} \\
& \leq \log{ \Big( 1+ C ( \beta, \vt^{n {-} 1} \omega, 1, T {+} 1)\Big)} +
2 \sup_{1 \leq \sigma \leq T +1} d_{H}( \varphi_{n-1 + \sigma}^{\pi}( \omega) f,
f).
  \end{align*}
Here we used again that \(\varphi^{\pi}_{s}(\omega)  f \) lies in \(E\) for all
\(\omega\) and \(s\), and that for \(g \in E\) we have \(m(g) \leqslant  1
\leqslant \| g \|_{\infty},\)
since it holds that \(\int_{\TT^{d}} g(x)\ud x = 1\). In fact this implies
\[ \log{\Big( m ( \varphi_{s}^{\pi}(\omega) f) \Big)} \leqslant 0 \leqslant \log{ \Big(
\| \varphi^{\pi}_{s}(\omega) f \|_{\infty} \Big),}
 \] 
so that
\begin{align*}
\log{ \Big( \| \varphi_{n {-} 1}^{\pi} (\omega) f \|_{\infty}\Big)} {-} \log{ \Big( m
\big(\varphi_{t}^{\pi} ( \omega) f \big) \Big)} & \leqslant \log{ \Big( \|
\varphi_{n {-} 1}^{\pi} (\omega) f \|_{\infty}\Big)} - \log{\Big( m
(\varphi_{n-1}^{\pi}(\omega) f) \Big)} \\
& \quad + \log{ \Big( \|
\varphi_{t}^{\pi}(\omega) f \|_{\infty} \Big)}  - \log{ \Big( m
\big(\varphi_{t}^{\pi} ( \omega) f \big) \Big)} \\
& \leqslant 2 \sup_{1 \leqslant \sigma \leqslant T + 1} d_{H}(\varphi_{n-1 +
\sigma}^{\pi}(\omega)f, f).
\end{align*}
Hence we have reduced \eqref{eqn:proof-synchro-holder-support-bound}
to proving the following:
  \begin{equation}\label{eqn:reduced_convergence_in_proof} 
    \limsup_{n \to \infty} \frac{1}{n} \left[ \log{\Big( C(\beta , \vt^{n {-} 1}
      \omega, 1 , T {+} 1)\Big)} + \sup_{1 \leq \sigma \leq T {+} 1}  d_{H} (
    \varphi_{ n {-} 1 {+} \sigma}^{\pi}( \omega) f, f) \right] \leq 0. 
  \end{equation} 
  Let us start with the last term and bound: 
\begin{equation}\label{eqn:proof-synchro-bound-complicated-term}
  \begin{aligned}
    d_{H} ( \varphi_{n {-} 1 {+} \sigma }^{\pi}& (\omega) f, f) \\
    & \leq \tau(\varphi_{\sigma}( \vt^{n {-} 1} \omega) )d_{H} ( \varphi_{n {-}
    1}^{\pi} (\omega) f, f) + d_{H}( \varphi_{\sigma}^{\pi}(\vt^{n {-} 1} \omega) f,
    f) \\ 
    & \leq \sum_{i = 0}^{n {-} 1} \prod_{j = i {+} 1}^{n {-} 1}
    \tau(\varphi_{1}( \vt^{j} \omega)) d_{H} ( \varphi_{1}^{\pi}( \vt^{i} \omega)
    f, f) + \sup_{1 \leq \sigma \leq T {+} 1} d_{H}( \varphi_{\sigma}^{\pi}
    ( \vt^{n {-} 1} \omega) f, f).
  \end{aligned}
\end{equation}
Here, in order to obtain the last inequality, we have iteratively applied
the following inequality, which holds for any $ j \in \NN $:
\begin{align*}
d_{H}(\varphi_{j+1}^{\pi}(\omega)f, f) & \leqslant
d_{H}(\varphi_{j+1}^{\pi}(\omega)f, \varphi_{1}^{\pi}(\vt^{j} \omega)f) +
d_{H}(\varphi_{1}^{\pi}(\vt^{j} \omega)f, f)\\
& \leqslant \tau(\varphi_{1}^{\pi}(\vt^{j} \omega))d_{H}(\varphi_{j}^{\pi}(\omega)f, f) +
d_{H}(\varphi_{1}^{\pi}(\vt^{j} \omega)f, f).
\end{align*}  
  By Assumption \ref{assu:properties_moments_of_solution_map}
  \( \EE [\sup_{1 \leq \sigma \leq T {+} 1} d_{H}( \varphi_{\sigma}^{\pi}
f, f)] < \infty\), hence by the ergodic theorem for all \(\omega \in
\widetilde{\Omega}\) (up to reducing \(\widetilde{\Omega}\)):
\[ \lim_{n \to \infty} \frac{1}{n} \sum_{i= 1}^{n}\sup_{1 \leqslant \sigma \leqslant T +1} d_{H}(
\varphi_{\sigma}^{\pi}(\vt^{i} \omega) f, f) = \EE \Big[ \sup_{1 \leq \sigma \leq T
{+} 1} d_{H}( \varphi_{\sigma}^{\pi} f, f) \Big] < \infty.\]
In particular, by Lemma~\ref{lem:convergence-series}, for any \(j \in \NN\):
  \begin{equation}\label{eqn:proof-synchro-convergence-to-zero-ergodic-theorem}
    \lim_{n \to \infty} \frac{1}{n} \sup_{1 \leq \sigma \leq T {+} 1}
    d_{H}( \varphi_{\sigma}^{\pi} ( \vt^{n {-} j} \omega) f , f) = 0.
  \end{equation}
Here we used that \( \widetilde{\Omega}\) is invariant under \(\vt\). So if
\(\omega \in \widetilde{\Omega}\), then also \(\vt^{-j} \omega \in
\widetilde{\Omega}\).
  Now observe that by Lebesgue dominated convergence, since \(d_{H}
(\varphi_{1}^{\pi} (\omega) f,f) \in L^{1}(\Omega)\) and since \(\tau( \cdot) \leqslant
1\) as well as \(\lim_{c \to \infty} \prod_{j =1}^{c}
\tau(\varphi_{1}(\vt^{j} \omega)) =0, \ \forall \omega \in
\widetilde{\Omega}\), it holds that:
  \[
    \lim_{ c \to \infty} \EE \bigg[ \prod_{j = 1}^{c} \tau(\varphi_{1}(
    \vt^{j} \cdot ))d_{H}( \varphi_{1}^{\pi}(\cdot) f, f)  \bigg] = 0.
  \]
  Hence fix any \(\ve>0\) and choose a deterministic \(c(\ve) \in \NN\) so that:
\begin{align*}
\EE \Big[\prod_{j = 1}^{c(\ve)} \tau(\varphi_{1}( \vt^{j} \cdot ))d_{H}( \varphi_{1}^{\pi}(\cdot) f, f) \Big] \leqslant \ve.
\end{align*}
  Now we use the bound~\eqref{eqn:proof-synchro-bound-complicated-term}
together with~\eqref{eqn:proof-synchro-convergence-to-zero-ergodic-theorem} and
the ergodic theorem to obtain:
  \begin{align*}
    \limsup_{n \to \infty} \frac{1}{n} & \sup_{1 \leq \sigma \leq T {+} 1}
    d_{H}( \varphi_{n {-} 1 {+} \sigma}^{\pi} (\omega) f, f) \\
    & \leq \limsup_{n \to \infty} \frac{1}{n} \sum_{i = 1}^{n {-} 1 {-}
c(\ve)} \prod_{j = i {+} 1}^{i {+} c(\ve)} \tau( \varphi_{1}( \vt^{j} \omega)) d_{H}(
    \varphi_{1}^{\pi}(\vt^{i} \omega) f, f) \\
& \quad + \limsup_{n \to \infty} \frac{1}{n} \sum_{i = n {-} 1 {-}
c(\ve)}^{n } \sup_{1 \leqslant \sigma \leqslant T + 1}  d_{H}(
\varphi_{\sigma}^{\pi}(\vt^{i-1} \omega) f, f)\\
& \leq \ve.
  \end{align*}
  As \(\ve\) is arbitrarily small we have proven that
\[ \limsup_{n \to \infty} \frac{1}{n}  \sup_{1 \leq \sigma \leq T {+} 1}  d_{H} (
    \varphi_{ n {-} 1 {+} \sigma}^{\pi}( \omega) f, f) \leqslant 0,\]
which is of the required order for~\eqref{eqn:reduced_convergence_in_proof}.
 To complete the proof of~\eqref{eqn:reduced_convergence_in_proof} we are left with the term containing
  \(C( \beta, \vt^{n} \omega, 1, T+1)\). Once more Assumption
  \ref{assu:properties_moments_of_solution_map} together with the ergodic
theorem and Lemma~\ref{lem:convergence-series} imply that:
  \[ 
    \lim_{n \to \infty} \frac{1}{n} \log{ C ( \beta, \vt^{n} \omega, 1, T {+}
    1)} = 0, 
  \] 
  thus completing the proof of \eqref{eqn:reduced_convergence_in_proof} and
hence of point \((i)\) of our theorem.

  \textit{Step 3.} Now, let us pass to the convergence in \(C^{\alpha}\)
backwards in time, which completes the proof of point \((ii)\). The proof is
analogous to, but simpler than the one we presented in Step \(2\).  The key
simplification consists in the fact that backwards in time the limit point \(
h_{\infty} (\omega, t) \) does not fluctuate (so the argument is essentially
deterministic, and does not rely on the law of large numbers), whereas forwards in time
all paths synchronize along the path \( h_{\infty}( \omega, n) \), whose
distribution does not vary with $ n $, but which fluctuates for fixed $ \omega
$, as $ n $ varies.

Since in
Equation~\eqref{eqn:support_time_inft_norm} we already proved convergence in the
\(\| \cdot \|_{\infty}\) norm, we now have to consider only the \([
\cdot]_{\alpha}\) seminorm. Up to replacing \(T\) with
  \(\lceil T \rceil\) assume \(T \in \NN\).  Then, consider \(n \in \NN\) with \(T<n {-} 1\) and \({-} T
  \leq t \leq T\) so that \(t = {-} T {-} 1 {+} \sigma\) with \( 1 \leq \sigma \leq
  2T {+} 1\). As in~\eqref{eqn:interpolation_in_proof} we define \(\theta =
\frac{\alpha}{\beta} \in (0,1)\) and use the interpolation bound of
Lemma~\ref{lem:interpolation-bound-general-regularity}:
  \begin{equation*}\label{eqn:interpolation_in_proof_backwards}
    \begin{aligned}
      & [h_{1}^{-n}(\omega, t) {-} h_{\infty}(\omega, t) {-} c(\omega, t,
      h^{1}_{0})]_{\alpha} \leq C( \alpha, \beta)\Big( \| h_{1}^{-n}( \omega, t) {-} h_{\infty}( \omega, t) {-}
	c(\omega, t, h^{1}_{0}) \|_{\infty}\Big)^{1 {-} \theta} \cdot \\
      & \qquad \qquad \qquad \qquad \qquad \qquad \qquad \qquad \qquad \cdot
      \Big( [ \log{ (\varphi_{n+t}^{\pi} (\vt^{-n}\omega)
	u^{1}_{0} )} ]_{\beta} {+} [ \log{(u_{\infty}(\omega, t) )} ]_{\beta}\Big)^{\theta}. 
    \end{aligned}
  \end{equation*}
Hence, in view of Equation~\eqref{eqn:support_time_inft_norm} it suffices to prove
that
\begin{equation*}
\begin{aligned}
\limsup_{n \to \infty} \left[  \frac{1}{n} \log{\bigg( \sup_{-T \leq t \leq  T } \Big(
	[\log{ (\varphi_{n+t}^{\pi} (\vt^{-n}\omega)
	u^{1}_{0} )} ]_{\beta} {+} [ \log{(u_{\infty}(\omega, t) )} ]_{\beta}
\Big) \bigg)}\right] \leqslant 0,
\end{aligned}
\end{equation*}
which we can further reduce to
\begin{equation}\label{eqn:proof-synchro-backwards-holder-estimate}
\begin{aligned}
\limsup_{n \to \infty} \left[  \frac{1}{n} \log{\bigg( \sup_{-T \leq t \leq  T }
      [\log{ (\varphi_{n+t}^{\pi} (\vt^{-n}\omega)
	u^{1}_{0} )} ]_{\beta} 
\bigg)}\right] \leqslant 0.
\end{aligned}
\end{equation}
Since the \([\cdot]_{\beta}\) seminorm is invariant under constant shifts (i.e.
\([ f + \zeta]_{\beta} = [f]_{\beta}\) for any \(f \colon \TT^{d} \to \RR\),
\(\zeta \in \RR\)),  and since $$ \log{ \big(
\varphi_{n+t}^{\pi}(\vt^{-n} \omega)u_{0}^{1} \big)} = \log{\big(
\varphi_{\sigma}(\vt^{-T-1} \omega) \circ \varphi^{\pi}_{n - T
-1}(\vt^{-n} \omega) u_{0}^{1} \big)  } +  \zeta(t, T, n, \omega,
u^{1}_{0}), $$ 
for some constant $ \zeta(t, T, n, \omega, u^{1}_{0}) \in \RR $,  we can rewrite the term inside the limit as
  \begin{align*}
    \Big[ \log{ \big( \varphi_{\sigma} (\vt^{{-} T {-} 1}\omega) \circ
    \varphi_{n {-} T {-} 1}^{\pi} ( \vt^{{-} n} \omega) u^{1}_{0} \big)}
\Big]_{\beta}.
  \end{align*}
At this point we want to exploit the regularising effect of $
\varphi_{\sigma} $, together with the fact that $ \varphi^{\pi}_{n - T - 1}
(\vt^{- n} \omega ) u_{0}^{1}$ is uniformly bounded in $ n $
(depending on $ \omega $).
 In fact, we observe that~\eqref{eqn:support_time_inft_norm} implies the
convergence \(\log\Big( \varphi_{n {-} T {-} 1}^{\pi} ( \vt^{{-} n}  \omega)
  u^{1}_{0} \Big) \to \log \Big( u_{\infty}( \omega, {-} T {-} 1)\Big)\) in \(C(\TT^{d})\) uniformly
  over \(u^{1}_{0}\). In addition, by the positivity of $ \varphi $ as in
  Assumption~\ref{assu:properties_moments_of_solution_map} and the fact that
  $ u_{\infty} $ is invariant under $ \varphi $, as in
  Theorem~\ref{thm:random_krein_rutman}, there exists
a $ c^{\prime} (\omega) > 0 $ such that $ u_{\infty}(\omega,-T-1)(x)
\geqslant c^{\prime} (\omega), \ \forall x \in \TT^{d}$. In particular, combining these
two observations we find a $ 0 < c(\omega) < c^{\prime} (\omega) $ and an $
n_{0} \in \NN $, such that
\[u_{\infty}( \omega, {-} T {-} 1)(x) \geqslant c(\omega), \quad 
\varphi_{n-T-1}^{\pi} ( \vt^{-n} \omega) u^{1}_{0} 
(x) \geq c(\omega) , \quad \forall x \in \TT^{d}, \  n \geqslant n_{0}, \  u^{1}_{0} \in
C(\TT^{d}).\] By point $ (ii) $ of Assumption
  \ref{assu:properties_moments_of_solution_map} we thus obtain
  \begin{align*} 
    \inf_{u^{1}_{0} \in C(\TT^{d})} \inf_{n> ( T {+} 1) \vee n_{0}} \inf_{\substack{ 1 \leq
    \sigma \leq 2T {+} 1 \\ x \in \TT^{d}}} \bigg[ \varphi_{\sigma}( \vt^{{-} T-1}\omega)
    \big( \varphi_{n -T{-} 1}^{\pi}( \vt^{{-} n} \omega) u^{1}_{0} \big)
(x)\bigg] & \geq \overline{c}(\omega),\\
\inf_{\substack{ 1 \leq
    \sigma \leq 2T {+} 1 \\ x \in \TT^{d}}} \bigg[ \varphi_{\sigma}( \vt^{{-} T-1}\omega)
    u_{\infty}(\omega, - T -1)  (x) \bigg] & \geq \overline{c}(\omega),
  \end{align*}
with
\[  \overline{c}(\omega) = c(\omega) \gamma( \vt^{{-} T-1} \omega,1, 2T {+} 1).\] 
Hence, applying Lemma~\ref{lem:derivatives-of-the-logarithm} together with the
regularising effect of $ \varphi_{\sigma} $ as in point $ (iii) $ of
Assumption~\ref{assu:properties_moments_of_solution_map}, we obtain for $ n
\geqslant n_{0} $:
\begin{align*}
\Big[ \log \big( \varphi_{\sigma}  (\vt^{{-} T {-} 1}\omega) \circ & \varphi_{n {-}
T {-} 1}^{\pi} ( \vt^{{-} n} \omega) u^{1}_{0} \big) \Big]_{\beta}  \leqslant
\left( \frac{1+ \big[ \varphi_{\sigma} (\vt^{{-} T {-} 1}\omega) \circ \varphi_{n
{-} T {-} 1}^{\pi} ( \vt^{{-} n} \omega) u^{1}_{0} \big]_{\beta}}{
\overline{c}(\omega)}\right)^{ \lfloor \beta \rfloor +1}  \\
& \leqslant  \left( \frac{1+C(\beta, \vt^{-T-1} \omega, 1, 2T+1)\| \varphi_{n {-} T {-} 1}^{\pi} ( \vt^{{-} n} \omega)
u^{1}_{0} \|_{\infty} }{
\overline{c}(\omega)} \right)^{\lfloor \beta \rfloor + 1} \\
& \leqslant  \left( \frac{1+ C(\beta, \vt^{-T-1} \omega, 1, 2T+1)M(\omega)}{
\overline{c}(\omega)}\right)^{\lfloor \beta \rfloor +1},
\end{align*}
with \(M(\omega) = \sup_{n \geqslant n_{0}} \|  \varphi_{n {-} T {-} 1}^{\pi} ( \vt^{{-} n}
\omega) u^{1}_{0}  \|_{\infty} < \infty\) in view
of~\eqref{eqn:support_time_inft_norm}. Hence~\eqref{eqn:proof-synchro-backwards-holder-estimate}
is proven, and this concludes the proof of the theorem.
\end{proof}

\section{Examples}

We treat two prototypical examples, which show the range of applicability of the
previous results. First, we consider the KPZ equation driven by a noise that is
fractional in time but smooth in space. In a second example, we consider the
KPZ equation driven by space-time white noise.

\subsection{KPZ driven by fractional noise}

Fix a Hurst parameter \(H \in \big( \frac{1}{2}, 1 \big)\) and consider
the noise \(\eta(t, x) = \xi^{H}(t)V ( x)\) for some
\(V \in C^{\infty}(\TT)\) and where \(\xi^{H}(t)= \partial_{t} \beta^{H}(t)\)
for a fractional Brownian motion \(\beta^{H}\) of Hurst parameter
\(H\). We restrict to \(H>\frac 1 2\) because the case $H=\frac{1}{2}$ is identical to the setting in
\cite{Sinai1991Buergers}, while for $H<\frac{1}{2}$ one encounters
difficulties with fractional stochastic calculus that lie beyond the scopes of this work.
For convenience, we let us define the noise \(\xi^{H}\) via its spectral
covariance function, see \cite[Section 3]{PipirasTaqqu2000Fbm}, namely as the
Gaussian process indexed by functions \(f  \colon \RR \to \RR\) such that
\(\int_{\RR} | \sigma|^{1 {-} 2 H}| \hat{f}( \sigma)|^{2} \ud \sigma < \infty\) (with
\( \hat{f}\) being the temporal Fourier transform), with covariance:
\begin{equation}\label{eqn:definition-covariance-function} 
  \EE \Big[ \xi^{H}(f) \xi^{H}(g) \Big] =
  c_{H} \int_{\RR} | \sigma|^{1 {-} 2H} \hat{f}(\sigma)
  \overline{\hat{g}(\sigma)} \ud \sigma, \qquad c_{H} = \frac{\Gamma(2 H {+} 1)
  \sin{(\pi H)} }{ 2 \pi}.  
\end{equation}
For the statement of the following lemma, recall the definition of
\(H^{\alpha}_{a}(\RR)\) given in~\eqref{eqn:definition-time-H-space}.
\begin{lemma}\label{lem:construction-noise-and-ergodicity}
Fix any \(H \in (\frac{1}{2}, 1), \ \alpha< H-1, \ a > \frac 1 2\). Let
\(\xi^{H}\) be the Gaussian process as defined by
  \eqref{eqn:definition-covariance-function}. Then, almost surely \(\xi^{H}\) takes
  values in \(H^{\alpha}_{a}(\RR)\). Next, define
  \(\Omega_{\mathrm{kpz}} = H^{\alpha}_{a}(\RR)\) and \(\mF= \mB(
  H^{\alpha}_{a}( \RR ))\) and let \(\PP\) be the law of \(\xi^{H}\) on
  \(\Omega_{\mathrm{kpz}}\).
Furthermore, let \( \{\vt^{z}\}_{z \in \ZZ}\) be the integer translation group,
which acts on smooth functions \(\varphi \in \mS(\RR)\) by:
  \[ \vt^{z} \varphi (t) = \varphi (t + z) , \quad \forall \ t \in \RR,\] 
and which is extended by duality to all distributions \(\omega \in
\Omega_{\mathrm{kpz}}\):
\[ \langle \vt^{z} \omega, \varphi \rangle  = \langle \omega, \vt^{-z} \varphi
\rangle, \quad \forall \varphi \in \mS(\RR).\]
 Then the space
  \( (\Omega_{\mathrm{kpz}}, \mF, \PP, \vt)\) forms an ergodic IDS.
  In addition, up to modifying \(\xi^{H}\) on a \(\vt-\)invariant null-set
  \(N_{0}\), for any \(\omega \in \Omega_{\mathrm{kpz}}\) there exists a \(\beta^{H}(\omega) \in
  C^{\alpha+1}_{\mathrm{loc}}(\RR)\) with: 
  \[\xi^{H}(\omega) = \partial_{t}
  \beta^{H}(\omega) \ \ \text{in the sense of distributions}, \qquad \beta^{H}_{0}(\omega)=0.\]
  Moreover, \( (\beta^{H}_{t})_{t \geq 0}\) has the law of a fractional
  Brownian motion of parameter \(H\).
\end{lemma}

\begin{proof}
  To show that \(\xi^{H}\) takes values in \(H^{\alpha}_{a}\) almost surely,
  observe that:
   \begin{align*}
     \EE \| \xi^{H} \|_{H^{\alpha}_{a}(\RR)}^{2} =  \sum_{j \geq -1}2^{2 \alpha j} \EE \|
     \Delta_{j} \xi (\cdot) / \langle \cdot \rangle^{a} \|_{L^{2}(\RR)}.
  \end{align*}
  Then one can bound:
  \begin{align*}
    \EE \Big[ \| \Delta_{j} \xi^{H}( \cdot) / \langle  \cdot
    \rangle^{a} \|_{L^{2}}^{2} \Big] & = \int_{\RR} \frac{1}{(1
    {+}|t|)^{2a} } \EE \big[ | \Delta_{j} \xi^{H} (t)|^{2} \big] \ud t
    \lesssim_{a} \sup_{t \in \RR} \EE \big[ | \Delta_{j} \xi^{H} (t)|^{2} \big] \\
    & = c_{H} \int_{\RR} | \sigma|^{1 {-} 2H} \varrho_{j}^{2}( \sigma)  \ud \sigma \lesssim 2^{j 2(1 {-} H)},
  \end{align*}
  where in the first line we used that \(2a > 1\). In the second line, we used that for \(j \geq 0\) \(\varrho_{j}( \cdot) =
  \varrho( 2^{{-} j} \cdot)\) for a function \(\varrho\) with support in an
  annulus (i.e. a set of the form \(\mA = \{ \sigma \ : \ A< | \sigma| <
  B\}\) for some \(0< A< B\)). This provides the required regularity estimate:
\[ \EE \| \xi^{H} \|_{H^{\alpha}_{a}(\RR)}^{2} < \infty. \] 

  The ergodicity is a consequence of the criterion in
  Proposition~\ref{prop:mixing_conditions} with \(\mathbf{B} =
H^{\alpha}_{a}(\RR)\), provided that we can verify
condition~\eqref{eqn:condition-vanishing-covariances} on the
  covariances. Observe that \(H^{\alpha}_{a}(\RR)\) is a separable Banach
space with dual \((H^{\alpha}_{a}(\RR))^* = H^{- \alpha}_{-a}(\RR)\) (this
result follows with the same calculations of \cite[Theorem 2.11.2]{Triebe2010}
for the unweighted case, see also the discussion in \cite[Section 7.2]{Triebe2010}), and that
the space \(\mS(\RR)\) of Schwartz functions, i.e.\ smooth functions with polynomial
decay at infinity of any order, is dense in \(H^{\beta}_{b}(\RR)\) for any
value of \(\beta \in \RR\) and \(b>0\) (see \cite[Remark 7.2.2]{Triebe2010}).\\
In view of these facts, and since we have shown that \(\EE \| \xi^{H}
\|_{\mathbf{B}}^{2} < \infty\), by
condition~\eqref{eqn:vanishing-covariances-simplified} of
Proposition~\ref{prop:mixing_conditions} it suffices to prove that for any
\(\varphi, \varphi^{\prime} \in \mS(\RR)\):
\[ \lim_{n \to \infty}  \mathrm{Cov}(\langle \xi^{H}, \varphi \rangle, \langle \vt^{n} \xi^{H},
\varphi^{\prime} \rangle) =0.\]
Here we can compute as follows:
  \begin{align*}
    \lim_{n \to \infty} \mathrm{Cov}( \langle \xi^{H}, \varphi \rangle, \langle \vt^{n} \xi^{H},
    \varphi^{\prime} \rangle) & \simeq \lim_{n \to \infty} \int_{\RR} | \sigma|^{1 {-} 2H}
    e^{i n \sigma} \hat{\varphi}( \sigma) \overline{\hat{\varphi}}^{\prime} ( \sigma) \ud \sigma\\
    & = 0. 
  \end{align*}
To obtain the last line we made use of the Riemann-Lebesgue lemma, since
\(f(\sigma) = | \sigma|^{1- 2H} \hat{\varphi}(\sigma)
\hat{\varphi^{\prime}}(\sigma)\) satisfies \(f \in L^{1}(\RR)\). In fact,
\(f\) is integrable near \(\sigma =0\) because \(H \in (1/2, 1)\) while
\(f(\sigma)\) decays polynomially fast for \(\sigma \to \pm \infty\) since
\(\varphi, \varphi^{\prime} \in \mS(\RR)\). Hence, ergodicity is proven.
  
  Now, one can define the primitive
  \(\beta^{H}(\omega)\) through
  \begin{align*}
    \beta_{t}^{H} = \xi^{H}(1_{[0,t]}), \ \ \text{in} \ \ L^{2}(\PP),
  \end{align*}
  so that following \cite[Section 3]{PipirasTaqqu2000Fbm} \(
  (\beta^{H}_{t})_{t \geq 0}\) has the
  law of a fractional Brownian motion. In particular, almost surely, the
  process \(\beta^{H}_{t} (\omega)\) has the required regularity. The null-set
  \( \overline{N}_0\) on which the result does not hold can be chosen to be
  \(\vt-\)invariant, by defining \(N_{0} = \bigcup_{z \in \ZZ} \vt^{z}
  \overline{N}_{0}\). Then one can set \(\xi^{H} = 0\) on \(N_{0}\).

\end{proof}
The next step is to show wellposedness of the SPDE:
\begin{equation}\label{eqn:she_fractional_noise}
  (\partial_{t} {-} \partial_{x}^{2}) u(t, x) = \xi^{H}(t) V(x) u (t, x),
  \qquad u(0,x) = u_0(x), \qquad (t,x) \in \RR_{+} \times \TT.
\end{equation}
We will work pathwise: since our noise is sufficiently regular, i.e.\ \(H
> \frac{1}{2}\) we can use Young integrals to make sense of the solution
(for \(H = \frac{1}{2}\), we would need It\^o integration instead). We will use
the following result:
\begin{lemma}\label{lem:young-integral}
For any \(\alpha, \beta, T>0\) such that \(\alpha+ \beta>1\) and \(f \in
C^{\alpha}([0, T]), \ g \in
C^{\beta}([0,T])\) one can define the Young integral
\[ \mathcal{I}_{t}(f, g) = \int_{0}^{t} f (s) \ud g(s).\] 
The map \(\mI\) is continuous between the spaces:
\[ \mI \colon C^{\alpha}([0, T]) \times C^{\beta}([0, T]) \to
C^{\beta}([0,T]),\]
satisfying the bound
\[ \| \mI(f, g) \|_{C^{\beta}([0,T])} \lesssim \| f \|_{C^{\alpha}
([0,T])} \| g \|_{C^{\beta}([0,T])}.\] 
If \(g \in C^{1}([0,T])\) the integral coincides with
\[ \mI_{t}(f, g) = \int_{0}^{t} f(s) \partial_{s} g(s) \ud s.\]
\end{lemma}
An instructive proof of this result is given in \cite[Proposition
6.11]{Friz-Victoir2010MultidimensionalRP} (for
\(\frac{1}{\alpha}-\)variation spaces instead of H\"older spaces), or in
\cite[Chapter 4]{Friz-Hairer2014CourseonRoughPaths}.

\begin{definition}\label{def:mild-solution-fractional-kpz}
  Consider \(H \in ( \frac 1 2 , 1)\) and let \(P_{t}\) be the periodic heat semigroup: \[P_{t} f (x) =
  \sum_{z \in \ZZ}(4 \pi t)^{{-} \frac d2}
\int_{\TT} f(y) e^{{-} \frac{|x {-} y - z |^{2}}{4t}}  \ud y.\] Fix
\(\omega \in \Omega_{\mathrm{kpz}}\) and \(\xi^{H}\) as in
Lemma~\ref{lem:construction-noise-and-ergodicity}. We
say that \(u \colon \Omega_{\mathrm{kpz}} \times \RR_{+} \times \TT \to \RR \) is a mild solution to
  Equation~\eqref{eqn:she_fractional_noise} if for any
\(\alpha < H\) and \(S>0\)
\[s \mapsto P_{t {-} s} [u(\omega, s, \cdot)  V(
  \cdot)](x) \in C^{\alpha}([S, t]), \qquad \forall t \geqslant S, \ x \in   \TT
\]
and if \(u\) satisfies:
\begin{equation}\label{eqn:midl-formulation-u} 
\begin{aligned}
& u( \omega, t, x)  = P_{t - S} u(\omega, S) (x)+ \int_{S}^{t} P_{t {-} s} [u(\omega, s, \cdot)  V(
\cdot)](x) \ud \beta_{s}^{H} ( \omega), \quad \forall t \geqslant S, \ x
\in \TT, \\
& \lim_{ S \to 0} u(\omega, S, \cdot)  = u_{0}(\cdot), \ \text{ in } \
C^{- \zeta}(\TT), \quad \forall \zeta>0,
\end{aligned}
\end{equation}
where, since the time regularities \(\alpha < H\) of the integrand and
\(\alpha^{\prime}< H\) of \(t \mapsto \beta^{H}_{t}(\omega)\) can be chosen so that
\(\alpha + \alpha^{\prime}>1\), because \(H \in (1/2, 1)\), the integral in~\eqref{eqn:midl-formulation-u} is
well-defined as a Young integral: see Lemma~\ref{lem:young-integral}.
\end{definition}

We can now prove the following result.
\begin{lemma}\label{lem:well-posedness-fractional-kpz}
  Consider \(H \in (\frac 1 2 , 1)\) and \(\Omega_{\mathrm{kpz}},\xi^{H}\) as in
  Lemma~\ref{lem:construction-noise-and-ergodicity}. For all \(\omega \in \Omega_{\mathrm{kpz}}\), for every \(u_0 \in C(\TT)\) there exists a unique
  mild solution \(u\) to Equation~\eqref{eqn:she_fractional_noise} such that
  for any \(\alpha<H, k \in \NN, 0 < S< T< \infty\):
  \[ (t, x) \mapsto \partial_{x}^{k} u(\omega, t,x) \in C^{\alpha}([S, T]
  \times \TT).\]
Moreover, the solution \(u\) can be represented as:
\begin{align*}
u(\omega, t, x) = e^{X(\omega, t, x)} w(\omega, t, x),
\end{align*}
with
\begin{equation}\label{eqn:for-X-fractional}
 X(\omega, t, x)  = \int_{0}^{t} P_{t {-} s} V (x) \ud \beta^{H}_{s}(\omega),
\end{equation}
and \(w\) a solution to
  \begin{equation}\label{eqn:for-w-fractional-noise}
  \begin{aligned} 
    (\partial_{t} {-} \partial_{x}^{2})w(t,x) & = 2 \partial_{x}X(t, x) \partial_{x} w (t , x) + (\partial_{x} X)^{2}(t, x)
    w(t, x), \\
    w(0, x) & = u_0(x). 
  \end{aligned}
  \end{equation}
  The solution map \((\varphi_{t}(\omega) u_{0}) (x) := u(\omega, t, x)\)
  defines a continuous linear RDS on \(C(\TT)\).
\end{lemma}
\begin{proof}

  Let us fix any \(\omega \in \Omega_{\mathrm{kpz}}\). Since all the
following arguments work pathwise, we will henceforth omit writing the
dependence on \(\omega\). To solve
Equation~\eqref{eqn:she_fractional_noise}, observe that \((s,x) \mapsto P_{t-s}
V(x) \in C^{\infty}([0,t] \times \TT)\), since \(V\) is smooth.
 We can then use Lemma~\ref{lem:young-integral} to define \(X(t ,x)\) by
Equation~\eqref{eqn:for-X-fractional}, so that formally \(X(t, x)\) solves: 
  \begin{align*}
    ( \partial_{t} {-} \partial_{x}^{2}) X(t, x) = \xi^{H}(t)V(x),
    \qquad X(0,x) = 0, \qquad \forall (t, x) \in \RR_{+} \times
    \TT.
  \end{align*} 
We will require a bound on the temporal regularity of \(X\). To this end, let us
write by integration by parts 
  \begin{align*}
    X(t, x) & =- \int_{0}^{t} \beta^{H}_{s} (P_{t {-} s} \partial_{x}^{2} V )(x) \ud s + V(x)
    \beta^{H}_{t},
  \end{align*}
  so that taking spatial derivatives in the above representation we obtain the
following regularity: 
\begin{equation}\label{eqn:proof-fractional-regularity-fro-X}
(t,x) \mapsto \partial_{x}^{k} X( t, x) \in
  C^{\alpha}([0, T] \times \TT)
\end{equation}
for any \( \alpha \in (\frac{1}{2}
, H), \ T >0, k \in \NN_{0}\).
We also observe that for any other path \(f \in C^{\alpha}([0,T]; \RR)\), by
Lemma~\ref{lem:young-integral} (taking smooth approximations of
\(\beta^{H}\) and using the continuity of the Young integral)
\begin{equation}\label{eqn:proof-fractional-ito-for-X} 
\int_{0}^{t} f_{s} \ud X(s, x) = \int_{0}^{t} f_{s} \partial_{x}^{2} X
(s,x) \ud s + \int_{0}^{t} f_{s} V(x) \ud \beta^{H}_{s}.
\end{equation}
  Now, as a consequence of Lemma~\ref{lem:schauder-regularization} there exists a
  unique mild solution \(w\) to Equation~\eqref{eqn:for-w-fractional-noise} and the same result implies that the
  solution \(w\) satisfies:
\begin{equation}\label{eqn:prof-fractional-regularity-w} 
(t, x) \mapsto \partial^{k}_{x}w( \omega, t , x) \in
    C^{1}_{\mathrm{loc}}((0,T] \times \TT),
\end{equation}
  for any \(T>0, k \in \NN_{0}\). 
  At this point, let us define \(u\) as \(u = e^{X}w\). For any fixed \(S>0\) we find that, by the chain rule
(which holds in view of Lemma~\ref{lem:young-integral}, by taking smooth approximations
of the integrand and integrator)
\begin{align*}
 u(t,x) & =u(S,x) + \int_{S}^{t} e^{X(s,x)} w(s,x) \ud X(s,x) +
\int_{S}^{t} e^{X(s,x)} w(s,x) \partial_{s} w(s,x) \ud s \\
& = u(S,x) +  \int_{S}^{t} \partial_{x}^{2} u(s, x) \ud s + \int_{S}^{t} u (s, x)
V(x) \ud \beta^{H}_{s},
\end{align*}
where we used~\eqref{eqn:proof-fractional-ito-for-X}
and~\eqref{eqn:for-w-fractional-noise}. Now by
\eqref{eqn:proof-fractional-regularity-fro-X} and
\eqref{eqn:prof-fractional-regularity-w} \[(t,x) \mapsto
\partial^{k}_{x} u(t,x) \in C^{\alpha}([S,T] \times \TT)\] for any \(k \in
\NN_{0}, \alpha \in (\frac{1}{2}, H)\) and \(0<S<T\). In particular, we find that
\[ (s,x) \mapsto \partial_{x}^{k} P_{t-s} [u (s, \cdot)V(\cdot)](x)
 \in C^{\alpha}([S,t] \times \TT), \quad \forall \ 0<S \leqslant t, \ \ \alpha
\in \Big(\frac{1}{2}, H \Big).\] 
Then we can define \( \widetilde{u}\) via the Young integral:
\[ \widetilde{u}(t,x) = P_{t - S} u(S,x) + \int_{S}^{t} P_{t - s} [u(s,
\cdot) V(\cdot)](x) \ud \beta^{H}_{s}.\]
An application of the chain rule show that \(u- \widetilde{u}\) is a smooth
solution to \( (\partial_{t} - \partial_{x}^{2}) (u- \widetilde{u}) = 0\), and
hence \(u = \widetilde{u}\). To conclude that \(u\) satisfies
Equation~\eqref{eqn:midl-formulation-u} we need that
\[ \lim_{S \to 0} u(S, \cdot) = u_{0}, \quad \text{ in } \ C^{-
\zeta}(\TT), \ \forall \zeta>0,\]
which follows since \(\lim_{S \to 0} w (S, \cdot) = u_{0}\). 
Conversely, one can follow the steps of this proof
backwards to find that every mild solution is of the required
form \(u=e^{X}w\).

  
  Finally, Lemma~\ref{lem:schauder-regularization} also implies that the
  solution map is, for fixed \(t \geq 0\), an element of \(\mL(C(\TT))\). To
  conclude we have to show that the cocycle property holds for \(\varphi\),
  namely that for \(n \in \NN_{0}\):
  \[ \varphi_{t + n}(\omega) u_{0} = \varphi_{t} (\vt^{n} \omega) \circ
  \varphi_{n}(\omega) u_{0}.\] 
  First observe that \(X_{t+ n}(\omega) - P_{t} X_{n}(\omega) =
  X_{t}(\vt^{n} \omega)\). Hence, recalling the decomposition of \(\varphi\):
  \[ \varphi_{t + n}(\omega) u_{0} = e^{X_{t}(\vt^{n} \omega)}
  (e^{P_{t} X_{n}(\omega)} w_{t+n}(\omega)),\]
  so that the cocycle property is proven since one can check that \(
    \overline{w}_{t}(\omega) = e^{P_{t}
  X_{n}(\omega)} w_{t+n}(\omega)\) solves
  Equation~\eqref{eqn:for-w-fractional-noise} with \(X(\omega)\) replaced by
  \(X (\vt^{n} \omega)\) and \( \overline{w}_{0} =
  \varphi_{t}(\omega)u_{0}\). 
\end{proof}

We can now prove that Equation \eqref{eqn:she_fractional_noise} falls in the
framework of the theory developed in the previous sections.

\begin{proposition}\label{prop:kpz_fractional_noise_rds_construction}

  The RDS \(\varphi\) introduced in
  Lemma~\ref{lem:well-posedness-fractional-kpz} satisfies, for any
  \(\beta>0\), Assumption \ref{assu:properties_moments_of_solution_map}.
  In particular, for all \(\omega \in \Omega_{\mathrm{kpz}}\), for any \(u_{0} \in C
  (\TT), u_{0}>0\), the function \(t \mapsto \log{(
  \varphi_{t}(\omega) u_{0} )}=: h_{t}(\omega)\) is the unique mild solution to
  \begin{equation}\label{eqn:kpz-fractiona-noise} 
    (\partial_{t} {-} \partial_{x}^{2}) h(\omega, t, x) = ( \partial_{x} h(
    \omega, t,x))^{2} + V(x)\xi^{H}(\omega, t), \qquad h(\omega, 0,x) = \log{
(u_{0}(x))},
  \end{equation}
meaning that for any \(\alpha < H, k \in \NN, 0 < S< T < \infty\):
  \[ (t, x) \mapsto \partial^{k}_{x} h(\omega, t,x) \in C^{\alpha}(( S, T) \times \TT) \] 
  and for all \(0 < S \leqslant t, \ \zeta>0\) and \(x \in \TT\):
\begin{align*}
& h(\omega, t, x) = P_{t-S} h (\omega,S)(x) + \int_{S}^{t} P_{t-s} [(\partial_{x}
h(\omega, s))^{2}](x) \ud s + \int_{S}^{t} P_{t-s}[V](x) \ud
\beta^{H}_{s}, \\
& \lim_{S \to 0} h(\omega, S, \cdot) = h_{0}( \cdot) \quad \text{ in } \quad
C^{-\zeta}.
\end{align*}
  Such solution satisfies all the results of
  Theorem~\ref{thm:synchronization_for_kpz}.

\end{proposition}

\begin{proof}
 Let us prove that $ \varphi $ satisfies
Assumption~\ref{assu:properties_moments_of_solution_map}. Points $ (i)-
(iii) $ of this assumption have to be verified for any $ \omega \in
\Omega_{\mathrm{kpz}} $: to lighten the notation
we will consider such $ \omega $ fixed and may not write explicitly the
dependence on it, as long as no confusion is possible.

\textit{(i).} Let us start with the kernel representation of $ \varphi $. Formally, one can
  write:
  \begin{equation}\label{eqn:dirac-in-flow}
    K(t, x, y) = \varphi_{t}( \delta_{y})(x).
  \end{equation}
  This can be made rigorous, if one can start
  Equation~\eqref{eqn:she_fractional_noise} in \(\delta_{y}\). In
  Lemma~\ref{lem:dirac-delta-continuity} we show that that for any \( \gamma>0, \
  \{\delta_{y}\}_{y \in \TT} \subseteq B^{- \gamma}_{1, \infty},\) and
  \(\| \delta_{x} - \delta_{y} \|_{B^{- \gamma}_{1, \infty}} \lesssim |x -
  y|^{\gamma}\). In addition, by Lemma~\ref{lem:well-posedness-fractional-kpz} the solution \(\varphi_{t}
  u_{0}= e^{X_t} w_{t},\) where \(X(t,x) = \int_{0}^{t} P_{t-s}[V](x) \ud
\beta^{H}_{s}\) does not depend on \(u_{0}\) and \(w\) is the solution to:
  \[ ( \partial_{t} - \partial_{x}^{2}) w  = 2 \partial_{x} X
  \partial_{x} w + (\partial_{x}X)^{2} w, \qquad w(0) = u_0.   \] 
  As the coefficients \( (\partial_{x} X)^{2}\) and \(\partial_{x}X\) are
  smooth in space and continuous in time, Lemma~\ref{lem:schauder-regularization} implies that the equation for
  \(w\) can be started also in \(u_{0} = \delta_{y}\). Let us denote with
\(w^{y}\) such solution. The same Lemma~\ref{lem:schauder-regularization}
implies the following bound, for any \(\eta \in [0,2), \ t \in [S, T]\)
and some \(q>0\):
\begin{align*}
 \| w^{y}(t, \cdot) - w^{z}(t, \cdot) \|_{B^{\eta- \gamma}_{1, \infty}}
\leqslant \| \delta_{y} - \delta_{z}\|_{B^{-\gamma}_{1, \infty}}e^{C(S,
T)(1 + \sum_{k = 0}^{ \lfloor \eta \rfloor + 1}
    \sup_{0 \leq t \leq T}\| \partial_{x}^{k} X_{t} \|_{\infty})^{q}}.
\end{align*}
We can choose \(\eta, \gamma\) so that \(\eta- \gamma>1\). In this case, by
Besov embeddings
\begin{align*} \|  w^{y}(t, \cdot) - w^{z}(t, \cdot)  \|_{C(\TT^{d})} & \lesssim \| w^{y}(t, \cdot) - w^{z}(t, \cdot)  
\|_{C^{\eta- \gamma -1}} \simeq \|  w^{y}(t, \cdot) - w^{z}(t, \cdot) 
\|_{B^{\eta - \gamma -1}_{\infty, \infty}}\\
&  \lesssim \| w^{y}(t, \cdot) - w^{z}(t, \cdot)  \|_{B^{\eta - \gamma}_{1,
\infty}}.
\end{align*} 
 Hence \(K\) in Equation~\eqref{eqn:dirac-in-flow} is rigorously defined as \( K(t, x, y) =
  e^{X(t,x)} w^{y}(t,x)\). In particular, putting together the previous bounds,
we have that
\begin{align*}
\sup_{S \leqslant t \leqslant T}\| K(t, \cdot, y) - K (t, \cdot, z) \|_{\infty} \leqslant
|y-z|^{\gamma}e^{C(S, T)(1 + \sum_{k = 0}^{ \lfloor \eta \rfloor + 1}
    \sup_{0 \leq t \leq T}\| \partial_{x}^{k} X_{t} \|_{\infty})^{q}},
\end{align*}
which implies that for any \(t>0, K(t) \in C(\TT \times \TT)\). That \(K\) is a
  fundamental solution for the PDE follows by linearity, thus concluding the
  proof of \((i)\) in
  Assumption~\ref{assu:properties_moments_of_solution_map}. 

\textit{(ii).} The fact that \(K\) is strictly positive, as required in point \((ii)\) of the assumptions
  is the consequence of a strong maximum principle (cf.
  \cite[Theorem 2.7]{Lieberman1996SecondOrderParabolicDEs}) applied to
\(w\), since \(e^{X} >0\). 

\textit{(iii).} The smoothing effect of point \((iii)\) in
  Assumption~\ref{assu:properties_moments_of_solution_map} follows again from
  the representation \( \varphi_{t} u_{0} = e^{X_{t}} w_{t}\) and the spatial
  smoothness of both \(X\) and \(w\), which we already showed in the proof of
  Lemma~\ref{lem:well-posedness-fractional-kpz}. 

\textit{(iv).} In particular, the just quoted smoothing
  effect can be made quantitative, via the estimate of
  Lemma~\ref{lem:schauder-regularization}, to obtain that for \(0< S<T <
  \infty\) there exist deterministic constants \(C(S, T), q \geq 0\) such that:
  \[ 
    \sup_{S \leq t \leq T}\| \varphi_{t}(\omega) u_{0} \|_{C^{\beta}} \leq 
    \| u_{0} \|_{\infty} e^{C(S, T)(1 + \sum_{k = 0}^{ \lfloor \beta \rfloor + 1}
    \sup_{0 \leq t \leq T}\| \partial_{x}^{k} X_{t}(\omega) \|_{\infty})^{q}}.  
  \]
  Note that at first Lemma~\ref{lem:schauder-regularization} allows to
  regularize at most by \(\eta< 2\), but splitting the interval \([0,S]\) into
  small pieces and applying iteratively the result on every piece provides the
  result for arbitrary \(\beta\).
  Now observe that in view of~\eqref{eqn:for-X-fractional} for any
\(k \in \NN\):
\begin{align*}
 | \partial_{x}^{k} X(t,x)| & = \bigg\vert \int_{0}^{t} P_{t-s}[
\partial_{x}^{k} V](x) \ud \beta^{H}_{s} \bigg\vert \\
& \lesssim \| s \mapsto P_{t-s} [\partial_{x}^{k} V](x) \|_{C^{\alpha}([0,T])} \|
s \mapsto \beta^{H}_{s} \|_{C^{\alpha}([0,T])},
\end{align*} 
for any \(\alpha \in (\frac{1}{2} , H)\), as an application of
Lemma~\ref{lem:young-integral}. Since \(s \mapsto P_{t-s}[
\partial_{x}V](x)\) is smooth (since \(V\) is smooth), we have obtained:
\begin{align*}
 \sup_{0 \leqslant t \leqslant T} \| \partial_{x}^{k}X (t) \|_{\infty}
\leqslant C(T,V) \| \beta^{H} \|_{C^{\alpha}([0,T])}.
\end{align*}
Now, for any \(q \geqslant 0\)
\[ \EE  \| \beta^{H} \|_{C^{\alpha}([0,T])}^{q} < \infty.\]
This follows from Kolmogorov's continuity criterion, or via calculations
similar to those in Lemma~\ref{lem:construction-noise-and-ergodicity}
(note that we show \( \EE \| \xi^{H} \|_{H^{\alpha}_{a}} < \infty\), but
similar calculations show that \(\EE \| \xi^{H} \|_{B^{\alpha, a}_{\infty,
\infty}}^{q}< \infty\), for any \(q \geqslant 0\)). We can conclude that:
  \[ \sum_{k = 1}^{\lfloor \beta \rfloor + 1} \EE \sup_{0 \leq t \leq T} \|
  \partial_{x}^{k} X_{t} \|_{\infty}^{q} < \infty,  \]
  thus proving the first average bound of point \((iv)\) in
  Assumption~\ref{assu:properties_moments_of_solution_map}.
 As for the second bound, in view of
  Lemma~\ref{lem:completeness_and_properties_positive_fcts}, one has:
  \begin{align*} 
    d_{H}( \varphi_{t}^{\pi}(\omega) f, f) & \lesssim \| \log{
      (\varphi_{t}(\omega) f)} - \log{ \int_{\TT}(\varphi_{t}(\omega) f) (x)}  \ud x\|_{\infty} + \|
   \log{ f} - \log{ \int_{\TT} f(x) \ud x}  \|_{\infty}\\
   & \lesssim \| \log{ ( \varphi_{t}(\omega) f)}  \|_{\infty} + \| \log{ f}
   \|_{\infty},
  \end{align*}
  so that our aim is to bound 
  \[ \EE \sup_{S \leq t \leq T} \| \log{ ( \varphi_{t} f)}
  \|_{\infty}.\] 
  On one side, one has the upper bound:
  \[ 
    \log{( \varphi_{t}(\omega) f) \leq \log{ \| \varphi_{t}(\omega) f
    \|_{\infty}} } \lesssim_{S, T} \log{ \| f \|_{\infty}} + \Big(1 +
      \sum_{k = 0}^{1} \sup_{0 \leq t \leq T}\| \partial_{x}^{k} X_{t}(\omega) \|_{\infty}
    \Big)^{q},  
  \] 
  which is integrable. As for the lower bound, observe that
  \(\log{ (\varphi_{t}(\omega) f)} = X_{t}(\omega) + \log{ w_{t}(\omega) }\).
  One can check that \(v_{t}(\omega) = \log{ w_{t}(\omega)}\) is a 
  solution to the equation:
  \begin{equation}\label{eqn:for-v} (\partial_{t} - \partial_{x}^{2}) v = 2( \partial_{x} X )
    \partial_{x} v + (\partial_{x} X)^{2}+ (
  \partial_{x} v)^{2}, \qquad v(0) = \log{f}. 
  \end{equation}
  By comparison (cf. \cite[Theorem 2.7]{Lieberman1996SecondOrderParabolicDEs}),
  one has: \(v(t,x) \geq - \| \log{ f } \|_{\infty}, \forall t \geq 0, x \in \TT\). So
  assuming that \(q \geq 1\), one has overall:
  \[ \| \log{( \varphi_{t}(\omega) f)}  \|_{\infty} \lesssim_{f} 1 + \Big(1 +
      \sum_{k = 0}^{1} \sup_{0 \leq t \leq T}\| \partial_{x}^{k} X_{t}(\omega) \|_{\infty}
  \Big)^{q}, \]
  which is once again integrable, completing the proof of $ (iv) $. 

We conclude that the required assumptions are satisfied
  and we can apply Theorem~\ref{thm:synchronization_for_kpz}.
  Finally, the fact that \(h_{t}\) satisfies the claimed smoothness assumption and is a mild
  solution to the KPZ equation driven by fractional noise follows by the same
steps of the proof of Lemma~\ref{lem:well-posedness-fractional-kpz}.

\end{proof}

\begin{remark}\label{rem:constants_for_fractional_noise}

  In the same setting as in
  Proposition~\ref{prop:kpz_fractional_noise_rds_construction}, for any \(h^1_{0}, h^2_{0} \in C(\TT)\) the
  constant \(c(\omega, t, h^{1}_{0}, h^{2}_{0})\) in
  Theorem~\ref{thm:synchronization_for_kpz} can be chosen
  independent of time.
\end{remark}

\begin{proof}
  Observe that it is
  sufficient to prove that there exists a constant \(
  \overline{c}(\omega, h^{1}_{0}, h^{2}_{0})\) such that for every \(\omega \in
  \widetilde{\Omega}\) (for an invariant set \( \widetilde{\Omega}\) of full
  \(\PP-\)measure) and any \(T >0\):
  \[ 
    \limsup_{n \to \infty} \frac{1}{n} \log{ \sup_{t \in [n, n {+} T]}} |
    c(\omega, t, h^{1}_{0}, h^{2}_{0}) {-} \overline{c}(\omega,
    h^{1}_{0}, h^{2}_{0}) | \leq \EE \log \big( \tau (\varphi_{1})\big).
  \] 
  As a simple consequence of Theorem~\ref{thm:synchronization_for_kpz} one has:
  \[ 
    \limsup_{n \to \infty} \frac{1}{n} \sup_{t \in [n, n {+}T]} \log \|
    \Pi_{\times} (h_{1}( \omega, t) {-} h_{2}(\omega, t) )\|_{\alpha} \leq  \EE
    \log \big( \tau (\varphi_{1})\big), 
  \] 
  for any \(\alpha > 0\), where \(\Pi_{\times}\) is defined  for \(f \in
  C(\TT)\) as \( \Pi_{\times} f = f {-} \int_{\TT} f(x) \ud x,\) and without
  loss of generality one can choose the constants to be defined as: 
  \[
    c(\omega, t, h^{1}_{0}, h^{2}_{0}) = \int_{\TT} h_{1}(\omega, t, x) {-}
    h_{2}(\omega, t, x) \ud x.
  \]
  Since by Proposition~\ref{prop:kpz_fractional_noise_rds_construction},
  \(h_{i}\) is a solution to the KPZ Equation one has that:
  \begin{equation}\label{eqn:constants_fractional_noise} 
    \partial_{t} c(\omega, t, h^{1}_{0}, h^{2}_{0}) = \int_{\TT}
    \partial_{x}(h_1 {-} h_2) \partial_{x} (h_1 {+} h_2 )( \omega, t, x) \ud x.
  \end{equation} 
  Now, in view of \((i)\) of Theorem
  \ref{thm:synchronization_for_kpz} one has:
  \[ \limsup_{n \to \infty} \frac{1}{n} \sup_{t \in [n, n {+} 1]} \log{ |
  \partial_{t}c (\omega, t, h^{1}_{0}, h^{2}_{0})|} \leq  \EE
    \log \big( \tau (\varphi_{1})\big).\]
  In particular this implies that there exists a constant \(
  \overline{c}(\omega, h^{1}_{0}, h^{2}_{0}) : = \lim_{t \to \infty} c(\omega,
  t, h^{1}_{0}, h^{2}_{0})\) and in addition
  \[ | \overline{c}( \omega, h^{1}_{0}, h^{2}_{0}) {-} c(\omega, t, h^{1}_{0},
  h^{2}_{0})| \leq \int_{t}^{\infty} | \partial_{s} c(\omega, s,
  h^{1}_{0}, h^{2}_{0})|\ud s \lesssim e^{{-} c t}, \]
  for any \(0< c < - \EE \log \big( \tau (\varphi_{1})\big)\), which proves the required result.

\end{proof}

\subsection{KPZ driven by space-time white noise}

In this section we consider the random force \(\eta\) in \eqref{eqn:intro_kpz}
to be space-time white noise \(\xi\) in one spatial dimension. That is, a Gaussian processes indexed by
functions in \(L^{2}( \RR \times \TT)\) such that:
\begin{equation}\label{eqn:space-time-wn-covariance}
  \EE \Big[ \xi(f) \xi(g) \Big] = \int\limits_{\RR \times \TT} f(t,x)g(t,x)
  \ud t \ud x.
\end{equation}
For the next result recall the definition of \(H^{\alpha}_{a}(\RR \times \TT)\)
from~\eqref{eqn:definition-space-time-H-space}.
\begin{lemma}\label{lem:space-time-wn-probability-space}
 Fix any \(\alpha< -1\) and \(a > \frac 1 2\). Let \(\xi\) be a Gaussian
process as defied in \eqref{eqn:space-time-wn-covariance}. Then, almost surely \(\xi\) takes values in
  \(H^{\alpha}_{a}(\RR \times \TT)\). In particular
\[ \EE \| \xi \|_{H^{\alpha}_{a}(\RR \times \TT)}^{2} < \infty. \] 
Next, define \(\Omega_{\mathrm{kpz}} = H^{\alpha}_{a} (\RR \times \TT), \mF = \mB(H^{\alpha}_{a}(\RR \times \TT))\)
  and let \(\PP\) be the law of \(\xi\) on \(\Omega_{\mathrm{kpz}}\).
Furthermore, let \( \{\vt^{z}\}_{z \in \ZZ}\) be the integer translation group,
which acts on smooth functions \(\varphi \in \mS(\RR \times \TT)\) by:
  \[ \vt^{z} \varphi (t, x) = \varphi (t + z, x) , \quad \forall \ (t,x) \in \RR \times \TT,\] 
and which is extended by duality to all distributions \(\omega \in
\Omega_{\mathrm{kpz}}\):
\[ \langle \vt^{z} \omega, \varphi \rangle  = \langle \omega, \vt^{-z} \varphi
\rangle, \quad \forall \ \varphi \in \mS(\RR \times \TT).\]
 Then the space
  \( (\Omega_{\mathrm{kpz}}, \mF, \PP, \vt)\) forms an ergodic IDS.
\end{lemma}

\begin{proof}
  We start by showing that \(\xi\) takes values in \(H^{\alpha}_{a}(\RR \times
  \TT)\) almost surely. By definition:
  \begin{align*}
    \EE \| \xi \|_{H^{\alpha}_{a}}^{2} =  \sum_{j \geq -1}2^{2 \alpha j} \EE \|
    \Delta_{j} \xi (\cdot) / \langle \cdot \rangle^{a} \|_{L^{2}},
  \end{align*}
  and for the latter one has:
  \begin{align*}
    \EE \Big[ \| \Delta_{j} \xi ( \cdot ) / \langle  \cdot
    \rangle^{a} \|_{L^{2}}^{2} \Big] & = \int_{\RR \times \TT } \frac{1}{(1
    {+}|t|)^{2a} } \EE \big[ | \Delta_{j} \xi (t, x)|^{2} \big] \ud t \ud x
    \lesssim_{a} \sup_{(t, x) \in \RR \times \TT} \EE \big[ | \Delta_{j} \xi
    (t,x)|^{2} \big]
    \\
    & = \int_{\RR \times \TT} |\mF^{-1}_{\RR \times \TT }\varrho_{j} |^{2}(t,
    x)  \ud t \ud x \simeq \sum_{k \in \ZZ} \int_{\RR}
    \varrho_{j}^{2}(k, \sigma) \ud \sigma \lesssim 2^{2j},
  \end{align*}
  where we used that \(2a > 1\) and that for \(j \geq 0\) \(\varrho_{j}( \cdot) =
  \varrho( 2^{{-} j} \cdot)\) for a function \(\varrho\) with support in
  an annulus. We can conclude that
\[ \EE \| \xi \|_{H^{\alpha}_{a}(\RR \times \TT)}^{2}< \infty. \] 
The last step in the proof is to show ergodicity of the IDS. Here we apply
  Proposition~\ref{prop:mixing_conditions}, so we have to check that
condition~\eqref{eqn:condition-vanishing-covariances}. We have proven
that \(\EE\| \xi \|_{H^{\alpha}_{a}}^{2} < \infty\), and (as in the proof of
Lemma~\ref{lem:construction-noise-and-ergodicity}) let us note that
\((H^{\alpha}_{a}(\RR \times \TT) )^*= H^{- \alpha}_{a}(\RR \times \TT)\) and
\(\mS(\RR \times \TT)\) is dense in \(H^{\beta}_{b}(\RR \times \TT)\) for every
\(\beta \in \RR, b>0\). Hence we can deduce ergodicity from the simplified
criterion~\eqref{eqn:vanishing-covariances-simplified}, namely we have to prove
that for \(\varphi, \varphi^{\prime} \in \mS( \RR \times \TT)\):
  \begin{align*}
    \lim_{n \to \infty} \mathrm{Cov}( \langle \xi, \varphi \rangle, \langle \vt^{n} \xi,
    \varphi^{\prime} \rangle ) = \lim_{n \to \infty}  \int_{\RR \times \TT} \varphi(t, x)
    \varphi^{\prime}(t - n, x) \ud t \ud x = 0,
  \end{align*}
  which is true because of the rapid decay at infinity of \(\varphi, \varphi^{\prime}\). This concludes the proof.
\end{proof}

Now we will consider \(h, u\) the respective solutions to the KPZ and
stochastic heat equation driven by space-time white noise:
\begin{align}
  (\partial_{t} {-} \partial_{x}^{2}) h &= (\partial_{x}h)^{2} {+} \xi {-}
  \infty, \qquad &h(0, x) = h_0(x), \qquad (t, x) \in \RR_{+} \times \TT
  \label{eqn:kpz_st_white_noise},\\
  (\partial_{t} {-} \partial_{x}^{2}) u & = u (\xi {-} \infty), \qquad &
  u(0,x) = u_0(x), \qquad (t,x) \in \RR_{+} \times \TT,
  \label{eqn:rhe_st_white_noise} 
\end{align}
in the sense of \cite[Theorem 6.15]{GubinelliPerkowski2017KPZ}.
Here the presence of the infinity ``\(\infty\)'' indicates the necessity of
renormalisation to make sense of the solution.
Wellposedness of the stochastic heat equation \eqref{eqn:rhe_st_white_noise}
can be proven also with martingale techniques, which do not provide a solution
theory for the KPZ equation, though. Instead, here we make use of pathwise
approaches to solving the above equations \cite{Hairer2013SolvingKPZ,
Hairer2014, GubinelliImkellerPerkowski2015}, that require tools such as regularity
structures or paracontrolled distributions. The main reference for us will be
\cite{GubinelliPerkowski2017KPZ}, which provides both a comprehensible
introduction (see for example Chapter 3) and a complete picture of the tools available in paracontrolled
analysis. Such theories consider smooth
approximations \(\xi_{\ve}\) of the noise \(\xi\), for which the equations are
well-posed, and study the convergence of the solutions as \(\ve \to 0\). The
renormalisation can then be understood as a Stratonovich-It\^o correction
term. We refer to the mentioned works as \textit{pathwise} approaches, since
they are completely deterministic, given the realization of the noise and some
functionals thereof. These functionals are collected in a random variable
called the \textit{enhanced noise} \(\YY(\omega)\).

In Lemma~\ref{lem:translation-of-enhanced-noise} we recall the construction of the
enhanced noise and record its transformation under \(\vt^{z}\).
Lemma~\ref{lem:translation-of-enhanced-noise} together with the existing
solution theory for the equation guarantee that the solution map forms a random
dynamical system. This is the content of the following result, which stands in
analogy to Lemma~\ref{lem:well-posedness-fractional-kpz} for fractional noise.



\begin{lemma}\label{lem:kpz_st_white_noise_solution_theory}

  Consider \((\Omega_{\mathrm{kpz}}, \mF, \PP)\) as in
  Lemma~\ref{lem:space-time-wn-probability-space}. Then for every \(\omega \in
  \Omega_{\mathrm{kpz}}\) and \(u_{0} \in C(\TT)\) there exists a unique
  solution \(u\) to Equation~\eqref{eqn:rhe_st_white_noise} in the sense of
  \cite[Theorem 6.15]{GubinelliPerkowski2017KPZ}, associated to the enhanced
  noise \( \YY(\omega)\) as in Lemma~\ref{lem:translation-of-enhanced-noise}
and the solution map \(\varphi_{t}(\omega) u_{0} = u_{t}\) defines a continuous
  linear RDS on \(C(\TT)\).
  

\end{lemma}

\begin{proof}
  Fix \(\omega \in \Omega\). The existence and uniqueness result \cite[Theorem 6.15]{GubinelliPerkowski2017KPZ}
  builds a solution to Equation~\eqref{eqn:rhe_st_white_noise} that depends continuously on the enhanced noise
\(\YY (\omega)\), and is continuous and linear with respect to initial
conditions $ u_{0} \in C(\TT^{d}) $ (in fact the theorem allows for $
u_{0} \in B^{- \beta}_{\infty, \infty} $ for $ \beta >0$ sufficiently small). The solution is
unique in a space of paracontrolled functions for which the product $
u \cdot(\xi - \infty) $ is defined in and appropriate pathwise sense.
  What we have to prove is that the solution map satisfies the cocycle property:
  \(\varphi_{n + t}(\omega) u_{0} = \varphi_{t}( \vt^{n} \omega) \circ
  \varphi_{n}(\omega) u_{0}\). From \cite[Theorem
4.5]{GubinelliPerkowski2017KPZ} (see the arguments that
  precede the theorem for a proof), the solution \(\varphi_{t+n}(\omega) u_{0}\) can be
represented as:
\begin{align*}
  \varphi_{t + n}(\omega) u_{0} = e^{Y_{t+ n}(\omega) + Y^{\TA}_{t+n}(\omega) +
  2Y^{\TB}_{t+ n}(\omega) } w^{P},
\end{align*}
where the terms inside the exponential are recalled in
Lemma~\ref{lem:translation-of-enhanced-noise}, and with \(w^{P}\) solving
\begin{align}\label{eqn:for-W-p}
  (\partial_{t} {-} \partial_{x}^{2}) w^{P} & = 4\Big[ (\partial_{t} {-}
    \partial_{x}^{2})(  Y^{\TD} {+} Y^{\TC}) + (\partial_{x}Y
      \partial_{x}Y^{\TB} {-} \partial_{x}Y \reso \partial_{x} Y^{\TB})\\
    & \qquad + \partial_{x} Y^{\TA} \partial_{x} Y^{\TB} + ( \partial_{x}Y^{\TB})^{2}
  \Big](\omega) w^{P} + 2 \partial_{x}(Y {+} Y^{\TA} {+} Y^{\TB})(\omega)
  \partial_{x} w^{P}, \nonumber\\
  w^{P}(0) & = e^{- Y_{0}(\omega) } u_{0}, \nonumber
\end{align}
in the paracontrolled sense of \cite[Theorem 6.15]{GubinelliPerkowski2017KPZ}.
Now one can use Equation~\eqref{eqn:translation-on-enhanced=noise} of
Lemma~\ref{lem:translation-of-enhanced-noise} to obtain:
\begin{align*}
  \varphi_{t + n}(\omega) = e^{Y_{t}(\vt^{n}\omega) +
  Y^{\TA}_{t}(\vt^{n} \omega) +  2 Y^{\TB}_{t}(\vt^{n} \omega)}
  \overline{w}^{P}_{t}(\omega),
\end{align*}
where 
\begin{align*}
  \overline{w}^{P}_{t}(\omega) = e^{P_{t} Y^{\TA}_{n}(\omega) + 2
  P_{t} Y^{\TB}_{n}(\omega)}w^{P}_{t + n}(\omega).
\end{align*}
In turn, \(\overline{w}^{P}(\omega)\) satisfies \(
\overline{w}^{P}_{0}(\omega) = e^{-Y_{0}(\vt^{n} \omega)} \varphi_{n}(\omega)
u_{0}\), and the proof of the cocycle property is complete if we show that \(
\overline{w}^{P} (\omega) \) is a solution to Equation~\eqref{eqn:for-W-p} with \(\omega\) replaced be
\(\vt^{n} \omega\) and initial condition \( e^{-Y_{0}(\vt^{n} \omega)}
\varphi_{n}(\omega) u_{0}\). If the enhanced noise $ \YY $ is a vector of
smooth functions the fact that $ \overline{w}^{P} $ solves the required equation is
immediate. Hence the claim follows by taking smooth approximations and using
the continuity of the solution theory in \cite[Theorem
6.15]{GubinelliPerkowski2017KPZ} with respect to the enhanced noise.

\end{proof}

The RDS \(\varphi\) introduced in the previous lemma falls into the framework
of Section~\ref{sec:syncrhonization-spdes}. To prove this, we follow the same
approach of
Proposition~\ref{prop:st_white_noise_satisfies_required_assumptions}, which
addresses the fractional noise case. First, we will construct the random kernel
\(K(\omega,t, x, y)\) for the solution map \(\varphi_{t}(\omega)\). Here the
key point is to use results from \cite{GubinelliPerkowski2017KPZ} to start
Equation~\eqref{eqn:rhe_st_white_noise} in \(u_0(x) = \delta_{y}(x)\). Then points \((i)-(iii)\) of Assumption~\ref{assu:properties_moments_of_solution_map}
follow by treating~\eqref{eqn:rhe_st_white_noise} as a pathwise perturbation of
the heat equation: these results have been already established, see e.g.
\cite{Cannizzaro2017Malliavin}. The most
challenging part of the proof is to prove the moments bounds of point
\((iv)\) of Assumption~\ref{assu:properties_moments_of_solution_map}. As in
Proposition~\ref{prop:kpz_fractional_noise_rds_construction} the proof of these
bounds relies on an appropriate decomposition \(\varphi_{t}(\omega)
u_{0} = e^{Z_{t}(\omega)} w_{t}(\omega)\), where \(Z_{t}\) is a functional of the noise,
together with a lower bound on \(w_{t}\) (first established in
\cite{PerkowskiRosati2018KPZonR}), which is the consequence of a comparison
principle.

\begin{proposition}\label{prop:st_white_noise_satisfies_required_assumptions}

  The RDS \(\varphi\) be defined as in Lemma
  \ref{lem:kpz_st_white_noise_solution_theory} satisfies
  Assumption \ref{assu:properties_moments_of_solution_map} for any \(\beta<
  \frac 12\). In particular, the results of Theorem
  \ref{thm:synchronization_for_kpz} apply.

\end{proposition}

\begin{proof}

  We will check, one by one, the requirements of
Assumption~\ref{assu:properties_moments_of_solution_map}. Since the first three
points are required to hold for every realization of the noise, let us fix \(\omega \in \Omega_{\mathrm{kpz}}\).  

  \textit{(i).} We start by checking the first property of
  Assumption~\ref{assu:properties_moments_of_solution_map}. We can define the kernel by \(K(\omega, t, x ,
  y) = \varphi_{t}(\omega)(\delta_{y})(x)\), where \(\delta_{y}\) indicates a
  Dirac \(\delta\) centered at \(y\). Here
  \(\varphi_{t}(\omega)(\delta_{y})\) is the solution to
  \eqref{eqn:rhe_st_white_noise} with \(u_{0}= \delta_{y}\). This solution
  exists in view of
  \cite[Theorem 6.15]{GubinelliPerkowski2017KPZ}: in fact, this result shows
  that for any \(0 < \beta, \zeta < \frac{1}{2},\) and any \(p \in [1, \infty]\)
  the solution map \(\varphi(\omega)\) can be extended to a map
  \[\varphi(\omega) \in C_{\mathrm{loc}}( (0, \infty); \mL( B^{- \gamma}_{p, \infty} ;
B^{\beta}_{p , \infty})),\]
where we used that, in the language of \cite{GubinelliPerkowski2017KPZ}, the space
\(\mD_{\mathrm{rhe}}^{\mathrm{exp}, \delta}\) of paracontrolled distributions,
in which the solution lives, embeds in \(C_{\mathrm{loc}}( (0, \infty); B^{\beta}_{p, \infty})\), for suitable values of \(\delta\)
as described in the quoted theorem. Near \(t=0\) one expects that
\(\| \varphi_{t} (\omega) u_{0} \|_{B^{\beta}_{p, \infty}}\) blows up, if
\(u_{0} \in B^{- \zeta}_{p. \infty}\). The exact speed of this blow-up is
provided as well in the theorem, but since we are not interested in quantifying the
blow-up, we can exploit the result we wrote to deduce the apparently stronger:
\begin{equation}\label{eqn:regularisation-rhe-1}
  \varphi (\omega) \in C_{\mathrm{loc}}( (0, \infty); \mL( B^{- \gamma}_{p, \infty};
C^{\beta})).\end{equation} 
This follows by Besov embedding: for \(p \leq q\): \(B^{\alpha}_{p, \infty}
(\TT^{d}) \subseteq B^{\alpha - d( \frac 1 p - \frac 1 q)}_{q,
\infty}(\TT^{d})\). Assuming without loss of generality that \(\beta,
\zeta> \frac 1 4\), uniformly over \(0 < S \leq t \leq T < \infty\) one can
bound:
\begin{align*}
  \| \varphi_{t} u_{0}\|_{C^{\beta}} \lesssim \| \varphi_{S/2}
  u_{0}\|_{C^{- \zeta}} \lesssim \| \varphi_{S/2} u_{0} \|_{B^{
  \beta}_{2, \infty}} \lesssim \| \varphi_{S/4} u_{0} \|_{B^{-
  \zeta}_{2, \infty}} \lesssim \| \varphi_{S/4} u_{0} \|_{B^{\beta}_{1,
  \infty}} \lesssim \| u_{0} \|_{B^{- \zeta}_{p, \infty}}.
\end{align*}
So overall we obtain \eqref{eqn:regularisation-rhe-1}, and in particular:
  \begin{equation}\label{eqn:regularisation-rhe-2} 
    \begin{aligned}
      \sup_{S \leq t \leq T} \| \varphi_{t}(\omega) u_{0} \|_{C^{\beta}}
      \leq C( \omega, \beta, \zeta, p, S , T) \| u_{0}
      \|_{B^{-\zeta}_{p , \infty}}, \ \ \text{for any} \ \ 0 < S < T
      < \infty.
    \end{aligned}
  \end{equation}
  Now since \( \{\delta_{y}\}_{y \in \TT} \subseteq B^{- \zeta}_{1}\) for any
  \(\zeta>0\), as proven in Lemma~\ref{lem:dirac-delta-continuity}, the kernel
  \(K(\omega, t, x, y)\) is well-defined. The continuity in \(t, x\) follows
  from the previous estimates, while the continuity in \(y\) follows from
  \eqref{eqn:regularisation-rhe-2} together with
  Lemma~\ref{lem:dirac-delta-continuity}.

  \textit{(ii).} We can pass to the second property of
  Assumption~\ref{assu:properties_moments_of_solution_map}. The upper bound \(\delta( \omega, S , T)\) is a 
  consequence of the continuity of the kernel \(K\). The lower bound
  \(\gamma(\omega, S, T)\) is instead a consequence of a strong maximum
  principle which, implies that \(K(\omega, t, x, y)>0, \forall t >0, x,y \in
  \TT\). In this pathwise setting, the strong maximum principle is proven in
  \cite[Theorem 5.1]{Cannizzaro2017Malliavin}
  (it was previously established in
  \cite{Mueller1991SupportHeatEquationWithNoise} with probabilistic techniques).

\textit{(iii).} The third property is a consequence of
Equation~\eqref{eqn:regularisation-rhe-2}, by defining \(C(\omega, \beta, S,
T):= C(\omega, \beta, \frac 1 4, \infty , S, T)\), so we are left with only the last
property to check.

\textit{(iv).} We start with the fact that
\[ \EE \log{ C( \beta, S, T)} < \infty. \] 
To see that this is the case, observe that there exists some deterministic
\(A( \beta,  S, T) , q \geq 1\) such that:
\begin{equation}\label{eqn:exponential-bound}
  \sup_{t \in [S, T]} \| \varphi_{t}(\omega) f \|_{ \beta} \leq
    e^{A( \beta, S, T )( 1 + \| \YY(\omega)
  \|_{\mY_{\mathrm{kpz}}} )^{q}} \| f \|_{C^{- \frac 1 4}},\end{equation}
  that is we can choose \(C(\omega, \beta,S, T)\):
\[ C(\omega, \beta, S, T) = e^{A( \beta, S, T )( 1 + \| \YY(\omega)
\|_{\mY_{\mathrm{kpz}}} )^{q}}.\]
Inequality~\eqref{eqn:exponential-bound} is implicit in the proof of \cite[Theorem
  6.15]{GubinelliPerkowski2017KPZ}, since the proof relies on a Picard
  iteration and a Gronwall argument. The bound can be found explicitly in 
  \cite[Theorem 5.5 and Section 5.2]{PerkowskiRosati2018KPZonR}: here the
  equation in set on the entire line \(\RR\), which is a more general setting,
  since one can always extend the noise periodically.
  Thus we have \(\EE \log{  C ( \beta, S , T)} \lesssim_{\beta, S, T} 1+ \EE
  \big[ \| \YY  \|_{\mY_{\mathrm{kpz}}}^{q} \big]\), so that the result is
  proven if one shows that for any \(q \geq 0\): \(\EE \| \YY
  \|_{\mY_{\mathrm{kpz}}}^{q}< \infty\), which is the content of
  \cite[Theorem 9.3]{GubinelliPerkowski2017KPZ}.

  We then pass to the second bound in $ (iv) $. Since by the triangle inequality the bound
  does not depend on the choice of \(f\), set \(f = 1\). It is thus enough to
  prove that:
  \[ 
    \EE \sup_{S \leq t \leq T} \| \log{ (\varphi_{t}( \omega) 1 )}  \|_{ \infty}
    < \infty.
  \]
  We proceed as in the proof of
  Proposition~\ref{prop:st_white_noise_satisfies_required_assumptions}. On one
  side one has the upper bound:
  \begin{align*}
    \log{( \varphi_{t}(\omega) 1)} \leq \log{ \| \varphi_{t}(\omega) 1
    \|_{\infty}} \leq \log{C(\omega, \beta, S, T)}, 
  \end{align*}
  which is integrable by the arguments we just presented. As for the lower bound, the approach of
  Proposition~\ref{prop:st_white_noise_satisfies_required_assumptions} has to
  be adapted to the present singular setting. One way to perform a similar
  calculation has been already developed \cite[Lemma 3.10]{PerkowskiRosati2018KPZonR}. We sketch again the
  argument here for clarity, assuming that the elements of \(\YY(\omega)\) are
  smooth. We will eventually refer to the appropriate wellposedness results to complete the proof. 
  Recall that \(\varphi_{t}(\omega) u_{0} = e^{Y_{t} (\omega) +
  Y^{\TA}_{t}(\omega) + 2 Y^{\TB}_{t}(\omega)} w^{P}_{t},\) where
  \(w^{P}\) solves Equation~\eqref{eqn:for-W-p}. Then define:
   \begin{align*}
    b(\YY) & = 2(\partial_{x} Y + \partial_{x} Y^{\TA} + \partial_{x}
    Y^{\TB})\\
    c(\YY) & = 4[(\partial_{t} {-} \partial_{x}^{2})(Y^{\TC} + Y^{\TD}) + (
	\partial_{x} Y \partial_{x} Y^{\TB} {-} \partial_{x} Y \reso
    \partial_{x} Y^{\TB}) + \partial_{x}Y^{\TA} \partial_{x} Y^{\TB} +
  (\partial_{x} Y^{\TB})^{2}]. 
  \end{align*}
  Assuming that \(b(\YY), c(\YY)\) are smooth one sees that \(h^{P} =
  \log{w^{P}}\) solves:
  \begin{align*}
    (\partial_{t} {-} \partial_{x}^{2}) h^{P} = b(\YY) \partial_{x}
    h^{P} + c(\YY) + (\partial_{x} h^{P})^{2}, \qquad h^{P}(0) =
    \log{w^{P}(0)}. 
  \end{align*}
  By comparison, \(h^{P} \geq - \widetilde{h}^{P}\), with the latter solving:
  \begin{align*}
    (\partial_{t} {-} \partial_{x}^{2}) \widetilde{h}^{P} = b(\YY) \partial_{x}
    \widetilde{h}^{P} - c(\YY) + (\partial_{x} \widetilde{h}^{P})^{2}, \qquad
    h^{P}(0) =- \log{w^{P}(0)}.
  \end{align*}
  In particular
  \[ h^{P} \geq - \log{ \widetilde{w}^{P}} \geq - \log{ \|
  \widetilde{w}^{P} \|_{\infty}},\] 
  where \( \widetilde{w}^{P}\) solves:
  \begin{align*}
    (\partial_{t} {-} \partial_{x}^{2}) \widetilde{w}^{P} = b(\YY)
    \partial_{x} \widetilde{w}^{P} - c(\YY) \widetilde{w}^{P}, \qquad
    \widetilde{w}^{P}(0) = \frac{1}{w^{P}(0)}. 
  \end{align*}
  Note that with respect to the equation in the proof of
  \cite[Lemma 3.10]{PerkowskiRosati2018KPZonR} some factors \(2\) are out of
  place: this is because here we consider the operator \(
  \partial_{x}^{2}\) instead of \( \frac{1}{2} \partial_{x}^{2}\). 
  The equation for \( \widetilde{w}^{P}\) is almost identical to the one for
  \(w^{P}\) and admits a paracontrolled solution
  as an application of \cite[Proposition 5.6]{PerkowskiRosati2018KPZonR}. In
  particular the quoted result implies that:
  \[ \sup_{S \leq t \leq T} \| \widetilde{w}_{t} \|_{\infty} \leq e^{\overline{C}(S,
  T)(1 + \| \YY \|_{\mY_{\mathrm{kpz}}})^{ \overline{q} }}, \]
  for some \( \overline{C}, \overline{q} \geq 1\). Since \( \| Y \|_{\infty}+
  \| Y^{\TA} \|_{\infty} + \| Y^{\TB} \|_{\infty} \lesssim \| \YY
\|_{\mY_{\mathrm{kpz}}},\) one has overall that:
\[ \log{ \varphi_{t}(\omega) 1} \gtrsim_{S, T} - 1 - \| \YY \|_{
\mY_{\mathrm{kpz}}}^{ \overline{q}}.\]
Together with the previous results and the moment bound on \(\EE \| \YY
\|^{ \overline{q}}\) we already recalled, this proves that:
\[ \EE \sup_{S \leq t \leq T} d_{H}( \varphi_{t} \cdot f, f) < \infty.\]
Hence the proof is complete.

\end{proof}

\begin{remark}\label{rem:control-problem}
  As an alternative to our proof of a lower bound to \(h_{t} = \log{
  \varphi_{t} (\omega) u_{0}}\), it
  seems possible to use an optimal control representation of
  \(h\), see \cite[Theorem 7.13]{GubinelliPerkowski2017KPZ}. Both
  approaches rely crucially on the
  pathwise solution theory for the KPZ equation.
\end{remark}




\begin{remark}

  In the previous proposition we have proven that we can apply Theorem
  \ref{thm:synchronization_for_kpz}. The latter guarantees synchronization up to
  subtracting time-dependent constants \(c (\omega , t)\). In fact it seems possible to choose
  \(c(\omega, t) \equiv \overline{c}(\omega)\) for a time-independent \(
  \overline{c} ( \omega)\).  For fractional noise we could show this in
  Remark~\ref{rem:constants_for_fractional_noise}, but in the argument we made use of the spatial
  smoothness of the noise to write an ODE for the constant \(c(\omega, t)\):
  Equation \eqref{eqn:constants_fractional_noise}.
 
  It seems reasonable to expect that the approach of Remark~\ref{rem:constants_for_fractional_noise} can be
  lifted to the space-time white noise setting by defining the product which
  appears in the ODE for example in a paracontrolled way. To complete the
  argument one would need to control the paracontrolled, and not only the
  H\"older norms in the convergences of Theorem
  \ref{thm:synchronization_for_kpz}. This
  appears feasible, but falls beyond the aims of the present paper. 

\end{remark}

\section{Mixing of Gaussian fields}\label{sec:mixing-of-gaussian-fields}

Let us state a general criterion which ensures that a possibly
infinite-dimensional Gaussian field is mixing (and hence ergodic). This is a
 generalization of a classical result for one-dimensional processes, cf.
\cite[Chapter 14]{SinaiCornfeldFomin1982ErgodicTheory}. We indicate with
\(\mathbf{B}^{*}\) the dual of a Banach space \(\mathbf{B}\) and write \(\langle
\cdot, \cdot \rangle\) for the dual pairing.

  \begin{proposition}\label{prop:mixing_conditions}

    Let \(\mathbf{B}\) be a separable Banach space. Let \(\mu\) be a Gaussian
    measure on \(( \mathbf{B}, \mB( \mathbf{B}))\) and \(\vt \colon \NN_{0}
    \times \mathbf{B} \to \mathbf{B}\) a dynamical system which leaves \(\mu\)
    invariant. Let \(\xi\) be any random variable with values in
\(\mathbf{B}\) and law \(\mu\). The condition
    \begin{equation}\label{eqn:condition-vanishing-covariances}
      \lim_{ n \to \infty} \mathrm{Cov}( \langle \xi, \varphi \rangle, \langle
      \vt^{n} \xi, \varphi^{\prime} \rangle) = 0, \qquad \forall \varphi,
      \varphi^{\prime} \in \mathbf{B}^{*} 
    \end{equation}
    implies that the system is mixing, that is for all \(A, B \in
    \mB(\mathbf{B})\):
    \[ \lim_{n \to \infty} \mu(A \cap \vt^{{-}n}B) = \mu(A) \mu(B). \] 
If in addition \(\mu\) satisfies that
\[ \EE [\| \xi \|_{\mathbf{B}}^{2} ] < \infty,\]
and \(S \subseteq \mathbf{B}^*\) is a dense subset then
\begin{equation}\label{eqn:vanishing-covariances-simplified}
 \lim_{n \to \infty} \mathrm{Cov}( \langle \xi, \varphi \rangle, \langle
\vt^{n} \xi, \varphi^{\prime}\rangle) = 0, \qquad \forall \varphi, \varphi^{\prime} \in
S 
\end{equation}
implies condition~\eqref{eqn:condition-vanishing-covariances}.
  \end{proposition}

  \begin{proof}

  First, we reduce ourselves to the finite-dimensional case. Indeed, note that
  the sequence \( (\xi, \vt^{n} \xi)\) is tight in \(\mathbf{B} \times
  \mathbf{B}\), because \(\vt\) leaves \(\mu\) invariant. Furthermore,
  tightness
  implies that the sequence is flatly concentrated (cf. \cite[Definition
  2.1]{Acosta1970}), that is for every \(\ve>0\) there exists a
  finite-dimensional linear space \(S^{\ve} \subseteq \mathbf{B} \times
  \mathbf{B}\) such that:
  \[ \PP\big( (\xi, \vt^{n} \xi) \in S^{\ve}\big) \geq 1 {-} \ve.\]
  Hence, it is sufficient to check the mixing property for \(A, B \in
  \mB(S^{\ve})\).

  This means that there exists an \(n \in \NN\) and \(\varphi_{i} \in
  \mathbf{B}^{*}\) for \(i = 1, \dots, n\) such that we have to check the mixing
  property for the vector: 
  \[ 
    ( (\langle \xi, \varphi_{i} \rangle)_{i = 1, \dots, n} , \ \ (\langle \vt^{n}
    \xi,  \varphi_{i} \rangle)_{i = 1, \dots, n}). 
  \] 
  In this setting and in view of our assumptions the result on the mixing
property follows from \cite[Theorem 2.3]{FuchsStelzer2013}. 

Finally, we have to prove that if \(\EE \| \xi \|_{\mathbf{B}}^{2} <
\infty\), then it suffices to check
condition~\eqref{eqn:vanishing-covariances-simplified} for \(\varphi,
\varphi^{\prime} \in S\). Indeed take any \(\psi, \psi^{\prime} \in
\mathbf{B}^*\). Since \(S\) is dense, consider for any \(\ve \in (0,1)\) a pair \(\varphi_{\ve},
\varphi^{\prime}_{\ve} \in S\) such that
\[ \| \psi - \varphi_{\ve} \|_{\mathbf{B}^*} + \| \psi^{\prime} -
\varphi^{\prime}_{\ve} \|_{\mathbf{B}^*} \leqslant \ve. \]
Then define \(M>0\) by 
\[ M = \sup_{\ve \in (0,1)} \left(  \| \varphi_{\ve} \|_{\mathbf{B}^*} + \|
\varphi^{\prime}_{\ve} \|_{\mathbf{B}^*} \right) < \infty. \] 
We can bound, for every \(n \in \NN\):
\begin{align*}
| \mathrm{Cov}(\langle \xi, \psi \rangle, \langle \vt^{n} \xi, \psi^{\prime}
\rangle) - & \mathrm{Cov}( \langle \xi, \varphi_{\ve} \rangle, \langle \vt^{n} \xi,
\varphi^{\prime}_{\ve} \rangle)| \\
& \leqslant \EE |\langle \xi, \psi - \varphi_{\ve} \rangle
\cdot \langle \vt^{n} \xi, \psi^{\prime} \rangle| + \EE | \langle \xi,
\varphi_{\ve} \rangle, \langle \vt^{n} \xi,
\psi^{\prime}- \varphi^{\prime}_{\ve} \rangle)|  \\
& \leqslant \ve \| \psi^{\prime} \|_{\mathbf{B}^*} \EE \| \xi
\|_{\mathbf{B}} \| \vt^{n} \xi \|_{\mathbf{B}} + \ve \| \varphi_{\ve}
\|_{\mathbf{B}^*} \EE \| \xi \|_{\mathbf{B}} \| \vt^{n} \xi
\|_{\mathbf{B}} \\
& \leqslant \ve \cdot 2  M \cdot \EE \| \xi \|_{\mathbf{B}}^{2}.
\end{align*}
In particular, since by assumption \(\varphi_{\ve}, \varphi_{\ve}^{\prime}\) satisfy
condition~\eqref{eqn:vanishing-covariances-simplified}, we have proven that:
\[ \limsup_{n \to \infty} | \mathrm{Cov}(\langle \xi, \psi \rangle, \langle
\vt^{n} \xi, \psi^{\prime} \rangle)| \leqslant \ve \cdot 2 M \cdot \EE \| \xi
\|_{\mathbf{B}}^{2}.\]  
As \(\ve\) is arbitrary this proves that
condition~\eqref{eqn:condition-vanishing-covariances} is true.
\end{proof}

\appendix\addtocontents{toc}{\protect\setcounter{tocdepth}{-1}}

\section{} 

\begin{lemma}\label{lem:schauder-regularization}
  Let \(P_{t}\) be the heat semigroup. One can estimate, for \(\alpha \in
  \RR, \beta \in [0,2), p \in [1, \infty]\) and any \(T>0\):
  \[ \sup_{0 \leq t \leq T } t^{\frac{\beta}{2} } \| P_{t} f
  \|_{B^{\alpha+ \beta}_{p, \infty}(\TT^{d})} \lesssim \| f
\|_{B^{\alpha}_{p, \infty}(\TT^{d})}.\]
In addition, if one chooses parameters $ \alpha, \gamma \in \RR $ such that for
some \(\beta \in [1, 2)\) and for $ \zeta : = \gamma \wedge \alpha + \beta $ it
holds that
\begin{align*}
\gamma + \zeta -1 >0,
\end{align*}
then, for any \(b \in L^{\infty}( [0,T] ;  B^{\gamma}_{\infty, \infty}(\TT^{d};
  \RR^{d})), c \in L^{\infty}( [0, T];  B^{\gamma}_{ \infty,
\infty}(\TT^{d}))\) and $ w_{0} \in B^{\alpha}_{p, \infty}(\TT^{d}) $, for any
$ p \in [1, \infty ) $, there exists a unique mild solution \(w\) to: 
  \[( \partial_{t} - \Delta) w(t,x) = b(t,x) \cdot \nabla w (t,x) +
  c(t,x)w(t,x), \qquad w(0,x) = w_0(x), \]
meaning that
\[ w(t,x) = P_{t} w_{0}(x) + \int_{0}^{t} P_{t-s} [b(s) \cdot \nabla w (s) +
c(s)w(s)](x) \ud s. \] 
  Moreover, there exists a \(q \geq 0\) and, for any time horizon $ T>0 $, a
constant \(C(T) > 0\) such that:
  \[ 
    \sup_{0 \leq t \leq T} t^{\frac{\beta}{2} }  \| w_{t} \|_{B^{\zeta}_{p, \infty}} \lesssim \|
    w_{0} \|_{B^{\alpha}_{p, \infty}} e^{C(T) ( 1+ \| b \|_{L^{\infty}( [0,T] ;
	B^{\gamma}_{\infty, \infty})} + \| c \|_{L^{\infty}( [0,T] ;
    B^{\gamma}_{\infty, \infty})})^{q}}.
  \]
\end{lemma}

\begin{proof}
  The estimate regarding the heat kernel is classical. For a reference from the
  field of singular SPDEs see \cite[Lemma
  A.7]{GubinelliImkellerPerkowski2015}. Let us pass to the PDE. Here consider
  any \(w\) such that \(M := \sup_{0 \leq t \leq T} t^{ \frac{\beta}{2} } \| w
  \|_{B^{\zeta}_{p, \infty}} < \infty\), and let \(N := \sup_{0 \leq t \leq T} \Big\{ \| b_{t}
    \|_{B^{\gamma}_{\infty, \infty}(\TT^{d}; \RR^{d})} + \| c_{t}
\|_{B^{\gamma}_{\infty, \infty}(\TT^{d})} \Big\}\). Then consider:
\[ \mI(w)_{t} = P_{t} w_{0} + \int_{0}^{t} P_{t-s} \Big[ b_{s} \cdot \nabla
w_{s} + c_{s} w_{s} \Big] \ud s. \] 
It follows from the smoothing effect of the heat kernel that:
\begin{align*}
  \sup_{0 \leq t \leq T} t^{\frac{\beta}{2} }   \| \mI(w)_{t} \|_{B^{\zeta}_{p,
  \infty}} & \lesssim \| w_{0} \|_{B^{\alpha}_{p, \infty}} + \sup_{0 \leq t
  \leq T} t^{\frac{\beta}{2} } 
  \int_{0}^{t} (t {-} s)^{ - \frac{\beta}{2} } \Big( \| b_{s} \cdot \nabla
    w_{s} \|_{B^{(\zeta-1) \wedge \gamma}_{p, \infty}} + \| c_{s} w_{s}
    \|_{B^{\zeta \wedge \gamma}_{p, \infty}}\Big)\ \ud s \\
\end{align*}
Now from our condition on the coefficient and estimates on products of
distributions (see \cite[Theorem 2.82 and
2.85]{BahouriCheminDanchin2011FourierAndNonLinPDEs}) the latter term can in
turn be bounded by:
\begin{align*}
      \sup_{0 \leq t \leq T} t^{\frac{\beta}{2} }   \| \mI(w)_{t} \|_{B^{\zeta}_{p,
  \infty}} & \lesssim \| w_{0} \|_{B^{\alpha}_{p, \infty}} + M N \sup_{0 \leq t \leq
    T} t^{\frac{\beta}{ 2} } \int_{0}^{t} (t {-} s)^{ - \frac{\beta}{2} } s^{- \frac{\beta}{2} } \ud s \\
    & \lesssim \| w_{0} \|_{B^{\alpha}_{p, \infty}} + M N T^{1-
    \frac{\beta}{2}}. 
\end{align*}
It follows that for small \(T>0\) the map \(\mI\) is a contraction providing
the existence of solutions for small times. By linearity and a Gronwall-type
argument, this estimate also provides the required a-priori bound.
\end{proof}

\begin{lemma}\label{lem:dirac-delta-continuity}
  For any \(\gamma>0\), the inclusion \(\{\delta_{y}\}_{y \in \TT^{d}}
  \subseteq B^{- \gamma}_{1, \infty}\) holds. Moreover, there exists an
  \(L(\gamma)>0\) such that:
  \[ \| \delta_{x} - \delta_{y} \|_{B^{- \gamma}_{1, \infty}} \leq L(\gamma) |x - y|^{\gamma}. \] 
\end{lemma}
\begin{proof}
  We divide the proof in two
  steps. Recall that by definition we have to bound \(\sup_{j \ge -1}2^{- \gamma
  j} \| \Delta_j (\delta_{x}-\delta_{y})\|_{L^1}\).
  Hence we choose $j_0$ as the smallest integer such that $2^{-j_0}
  \le |x - y|.$ We first look at small scales $j \ge j_0$ and then
  at large scales $j < j_0.$ For small scales, by the Poisson summation formula,
  since \(\varrho_{j}(k) = \varrho_{0}(2^{{-} j} k)\), and by defining
  \(K_{j}(x) = \mF_{\RR}^{{-} 1} \varrho_{j} (x)= 2^{j}K(2^j x)\) for some
  \(K \in \mS(\RR)\) (the space of tempered distributions):
  \begin{align*}
    2^{- \gamma j}\big\| \Delta_j (\delta_{x}- \delta_{y})
    \big\|_{L^1} & \le |x - y|^{\gamma} \int_{\RR} \  2^j  | K (2^j(z-x)) -
    K (2^j(z-y))| \ud z \\
    & \lesssim |x - y|^{\gamma} \int_{\RR} 2^j
    |K(2^j z)| \ud z \lesssim |x-y|^{\gamma}.
  \end{align*}
  While for large scales, since we have $|2^j(x {-} y)| \le 1$, applying the
  Poisson summation formula, by the mean value theorem and since \(K \in
  \mS(\RR)\) (the Schwartz space of functions):
  \begin{align*}
    2^{- \gamma j}\big\| \Delta_j (\delta_{x}- \delta_{y}) \big\|_{L^1} & \le
    2^{- \gamma j} \int |K (z) - K (z+ 2^j(x{-} y))| \ud z\\
    & \le  |x-y|^{\gamma} \int \max_{| \xi - z | \le 1}\frac{| K(\xi) -
    K(z)|}{|\xi - z |^{\alpha}} \ud z  \lesssim |x-y|^{\gamma}.
  \end{align*}
  The result follows.

\end{proof}

\begin{lemma}\label{lem:translation-of-enhanced-noise}
  Fix any \(\alpha < \frac 1 2\). Consider the space
  \[ 
    \mY_{\mathrm{kpz}} \subseteq C_{\mathrm{loc}}([0, \infty); C^{\alpha}
    \times C^{2 \alpha}\times C^{\alpha+ 1} \times C^{2\alpha+ 1} \times C^{2\alpha+ 1}
    \times C^{2 \alpha - 1}), 
  \] 
  with the norm \(\| \cdot \|_{\mY_{\mathrm{kpz}}}\) as in \cite[Definition
  4.1]{GubinelliPerkowski2017KPZ}. There exists a random variable \(\YY \colon
  \Omega_{\mathrm{kpz}} \to \mY_{\mathrm{kpz}}\) which coincides almost surely
  with the random variable constructed in \cite[Theorem
  9.3]{GubinelliPerkowski2017KPZ} and is given by:
  \[ \YY(\omega) = (Y(\omega), Y^{\TA}(\omega), Y^{\TB}(\omega),
  Y^{\TC}(\omega), Y^{\TD}(\omega), \partial_{x} \mathcal{P} \reso \partial_{x}Y (\omega)), \]
  where the latter solve (formally):
  \begin{align*}
    (\partial_{t} - \partial_{x}^{2}) Y &= \Pi_{\times} \xi,\\
    (\partial_{t} - \partial_{x}^{2}) Y^{\TA} &= (\partial_{x}
    Y)^{2} - \infty,\\
    (\partial_{t} - \partial_{x}^{2})Y^{\TB} & =  \partial_{x}Y
    \partial_{x} Y^{\TA},\\
    (\partial_{t} -\partial_{x}^{2}) Y^{\TC} &= \partial_{x} Y^{\TB} \reso
    \partial_{x}Y - \infty,\\
    (\partial_{t} - \partial_{x}^{2}) Y^{\TD} & = (\partial_{x} Y^{\TA}
    )^{2} - \infty\\
    (\partial_{t} - \partial_{x}^{2}) \mathcal{P} &= \partial_{x}Y.
  \end{align*}
  Here \(\Pi_{ \times } f = f - \int f(x) \ud x\) and \(f \reso g =
  \sum_{|i - j| \leq 1} \Delta_{i}f \Delta_{j} g\) is the resonant product
  between two distributions (which is a-priori ill-defined). Finally, the
  presence of infinity indicates the necessity of Wick renormalisation, in the
  sense of \cite[Theorem 9.3]{GubinelliPerkowski2017KPZ}. \(Y\) is started in
  invariance, that is:
  \[ Y_{t} = \int_{- \infty}^{t} P_{t-s} \Pi_{\times} \xi \ud s, \]
  while all other elements are started in \(Y^{\tau}(0) =0\). In particular
  \(\YY\) changes as follows under the action of \(\vt^{n}\), for \(n \in
  \NN_{0}, t \geq 0, \omega \in \Omega_{\mathrm{kpz}}\):
  \begin{equation}\label{eqn:translation-on-enhanced=noise}
  \begin{aligned}
    \YY_{t}(\vt^{n} \omega) = ( & Y_{t+n} , Y^{\TA}_{t+n} {-}
      P_{t}Y_{n}, Y^{\TB}_{t+n} {-} P_{t} Y^{\TB}_{n}, \\
      & Y^{\TC}_{t+n } {-} 
    P_{t} Y_{n}^{\TC}, Y^{\TD}_{t +n} - P_{t} Y^{\TD}_{n}, \partial_{x}
      (\mathcal{P}_{t+n} {-}   P_{t} \mP_{n}) \reso \partial_{x} (Y_{t+n} {-}
    P_{t} Y_{n}))(\omega).
  \end{aligned}
\end{equation}
\end{lemma}

\begin{proof}
  The only point that requires a proof is the action of the translation
  operator. By taking into account the initial conditions and using \cite[Theorem 9.3]{GubinelliPerkowski2017KPZ},
  Equation~\eqref{eqn:translation-on-enhanced=noise} holds for fixed
  \(n\), for all \(\omega \not\in N_{n}\) and all \(t \geq 0\), for a given null-set
  \(N_{n}\) (since the random variables are constructed in
  \(L^{2}(\Omega_{\mathrm{kpz}}; \mY_{\mathrm{kpz}})\)). Considering
    \(N = \bigcup_{n \in \NN} N_{n}\) and setting \(\YY(\omega) =0\) for
    \(\omega \in N,\) one obtains the result for all \(\omega \in
    \Omega_{\mathrm{kpz}}\).
\end{proof}

\begin{lemma}\label{lem:convergence-series}
Consider a sequence \(\{a_{k}\}_{k \in \NN}\) of positive \( (a_{k} \geqslant
0)\) real numbers. Suppose that
\[ S_{n} = \frac{1}{n} \sum_{k =1}^{n} a_{k}\]
converges, namely that there exists a \(\sigma \in [0, \infty)\) such that
\[ \lim_{n \to \infty} S_{n} = \sigma.\]
Then
\[ \lim_{n \to \infty} \frac{1}{n} a_{n} =0. \]  
\end{lemma}
\begin{proof}
Since \(S_{n}\) is convergent fix any \(\ve>0\) and let \(n(\ve) \in \NN\) be
such that 
\[ |S_{n} - S_{m}| \leqslant \ve, \quad \forall n,m \geqslant n(\ve).\]
We can assume, up to taking a larger \(n(\ve)\), that \(n(\ve) \geqslant
\frac{\sigma + \ve}{\ve}\). Now consider \(n \geqslant n(\ve) + 1\). We can
compute
\begin{align*}
\ve \geqslant |S_{n} - S_{n-1}| & = \Big\vert \frac{a_{n}}{n} - \left(  \frac{1}{n-1}
-\frac{1}{n} \right) \sum_{k = 1}^{n-1} a_{k}  \Big\vert \\
& = \Big\vert \frac{a_{ n}}{ n} - \frac{1}{n} S_{n-1} \Big\vert \\
& \geqslant \frac{a_{n}}{n} - \frac{S_{n-1}}{n} \geqslant  \frac{a_{n}}{n} -
\ve \frac{\sigma + \ve}{ \sigma + \ve},
\end{align*}
which implies that \( \frac{a_{n}}{n} \leqslant 2 \ve \) for all
\(n \geqslant n (\ve)+1\). Since \(\ve\) is arbitrary this completes the proof.
\end{proof}

\begin{lemma}\label{lem:interpolation-bound-general-regularity}
Consider any \(\beta \in (0, \infty), \ \alpha \in (0, \beta)\) and let \(\theta = \frac{\alpha}{\beta} \in (0,1)\). Then there exists a
constant \(C(\alpha, \beta)> 0\) such that for every \(f \in C^{\beta}\):
\[ [f]_{\alpha} \leqslant C(\alpha, \beta) \| f \|_{\infty}^{1-
\theta}[f]_{\beta}^{\theta}.\] 
\end{lemma}

\begin{proof}
We start by recalling, for \(k \in \{1, \ldots, d\}\), the one-dimensional Landau-Kolmogorov inequality
(see for example \cite{Kolmogorov1949Interpolation} or many online resources):
\begin{align*}
\| \partial_{x_{k}} f \|_{\infty} \lesssim \| f \|_{\infty}^{1 -
\frac{1}{n}} \| \partial_{x_{k}}^{n} f \|_{\infty}^{\frac{1}{n}},
\end{align*}
Iterating this inequality one obtains that for any \(n,l \in \NN\) and
\(k_{i} \in \{1, \ldots, d\}, \ \forall i =1, \ldots, l\):
\begin{align*}
\| \partial_{x_{k_{1}}} \cdots \partial_{x_{k_{l}}} f \|_{\infty} & \lesssim
\| \partial_{x_{k_{2}}} \cdots \partial_{x_{k_{l}}} f
\|_{\infty}^{1 - \frac{1}{n+1}} [ f ]_{n+l}^{\frac{1}{n+1}} \\
& \lesssim \| f \|_{\infty}^{\prod_{i = 1}^{l}\big(1- \frac{1}{n+i}\big)} [
f ]_{n+l}^{\sum_{i=1}^{l} \frac{1}{n+i} \prod_{j=1}^{i-1}\big( 1-
\frac{1}{n+j} \big)}.
\end{align*}
Since (both identities can be proven by induction over \(l\)):
\begin{align*}
\prod_{i=1}^{l}\Big( 1- \frac{1}{n+i} \Big)= 1 - \frac{l}{n+l}, \qquad
\sum_{i=1}^{l} \frac{1}{n+i} \prod_{j=1}^{i-1}\Big( 1-
\frac{1}{n+j} \Big)= \frac{l}{n+l},
\end{align*}
we have proven that
\begin{equation}\label{eqn-proof-interpolation-integer}
[ f ]_{l} \leqslant  C(l,n+l) \|
f \|_{\infty}^{1 - \frac{l}{n+l}} [
f ]_{n+l}^{\frac{l}{n+l}},
\end{equation}
which is the desired inequality for integer \(\alpha, \beta\). To pass to the
fractional case we will first prove that for \(\beta \geqslant n, n \in \NN\):
\begin{equation}\label{eqn:proof-interpolation-fractional-first}
[f]_{n} \lesssim \| f \|_{\infty}^{1 - \frac{n}{\beta}}
[f]_{\beta}^{\frac{n}{\beta}}.
\end{equation}
We can further simplify this by considering \(\beta \in (1, 2)\) and proving:
\begin{align*}
\| \partial_{x_{k}} f \|_{\infty} \leqslant 2 \| f \|_{\infty}^{1 -
\frac{1}{\beta}} [f]_{\beta}^{\frac{1}{\beta}}
\end{align*}
To obtain this let \(e_{k}\) be the unit vector in the \(k-\)th direction, and
consider for \(h > 0, x \in \TT^{d}\):
\begin{align*}
\partial_{x_{k}} f(x) = \frac{f(x + h e_{k}) -f(x)}{h} + R(x,h).
\end{align*}
Since
\[  \frac{f(x + h e_{k}) -f(x)}{h} = \partial_{x_{k}} f(\xi),\]
for some \(\xi \in [x, x+h e_{k}]\) (were \([x, x+h e_{k}]\) is the line
between \(x\) and \(x +h e_{k}\)), we can bound the rest term by:
\[ |R(x,h)| \leqslant \sup_{\xi \in [x, x+h e_{k}]} | \partial_{x_{k}} f
(x) - \partial_{x_{k}}f(\xi)| \leqslant h^{\beta-1}[ f
]_{\beta}.\]
Hence we have
\begin{align*} \| \partial_{x_{k}} f \|_{\infty} &  \leqslant h^{-1} \| f \|_{\infty} +
h^{\beta-1}[f]_{\beta}\\
& \leqslant 2\| f \|_{\infty}^{1 - \frac{1}{\beta}}
[f]_{\beta}^{\frac{1}{\beta}} 
\end{align*}
by setting \(h =  (\| f \|_{\infty}/[f]_{\beta})^{\frac{1}{\beta}}\). Next, we
deduce~\eqref{eqn:proof-interpolation-fractional-first} for 
any \(\beta > 1\) and \(n= \lfloor \beta \rfloor\). Using all the estimates we already derived:
\begin{align*}
[f]_{n} & \lesssim [f]_{n-1}^{1 - \frac{1}{\beta -n
+1}}[f]_{\beta}^{\frac{1}{\beta - n +1}} \\
& \lesssim \| f \|_{\infty}^{\big( 1 - \frac{n-1}{n} \big) \big( 1 -
\frac{1}{\beta - n+1}\big)} [f]_{n}^{\big( 1 - \frac{1}{\beta- n +1} \big)
\frac{n-1}{n}} [f]_{\beta}^{\frac{1}{\beta- n +1}} \\
\iff [f]_{n}^{\zeta} & \lesssim \| f \|_{\infty}^{\big( 1 - \frac{n-1}{n} \big) \big( 1 -
\frac{1}{\beta - n+1}\big)} [f]_{\beta}^{\frac{1}{\beta- n +1}},
\end{align*}
for 
\[ \zeta = 1 - \Big( \frac{\beta - n}{\beta - n +1} \Big) \Big(
\frac{n-1}{n} \Big) = \frac{\beta}{n(\beta-n+1)},\]
so that the last estimate implies
\eqref{eqn:proof-interpolation-fractional-first} for the chosen \(\beta\) and
\(n\). To conclude the proof
of~\eqref{eqn:proof-interpolation-fractional-first} we have to consider the
case \(\beta >1, n \leqslant \lfloor \beta \rfloor\). We find that:
\begin{align*}
[f]_{n} & \lesssim \| f \|_{\infty}^{1 - \frac{n}{\lfloor \beta \rfloor}}
[f]_{\lfloor \beta \rfloor}^{\frac{n}{\lfloor \beta \rfloor}} \\
& \lesssim \| f \|_{\infty}^{1 - \frac{n}{\lfloor \beta \rfloor}} \Big( \|
f \|_{\infty}^{1 - \frac{\lfloor \beta \rfloor}{\beta}} [f]_{\beta}^{
\frac{\lfloor \beta \rfloor}{\beta}} \Big)^{\frac{n}{\lfloor \beta \rfloor}}\\
& \lesssim \| f \|_{\infty}^{1 - \frac{n}{\beta}}
[f]_{\beta}^{\frac{n}{\beta}}.
\end{align*}
At this point, we can collect all our results to complete the proof. Consider
\(k, n \in \NN_{0}\) such that \(\alpha \in [k, k+1)\) and \(\beta \in [n, n+1)\).
Of course \(n \geqslant k\). Furthermore, define
\[ \alpha^{\prime} = \alpha - k, \qquad \beta^{\prime} = \beta - n.\]
\textit{Step 1: \(n =k\).} Note that one can bound
\begin{align*}
[f]_{\alpha} \simeq \sum_{| \mu | = k} \sup_{x \neq y} \frac{|
\partial^{\mu} f(x)- \partial^{\mu} f(y)|}{|x-y|^{\alpha^{\prime}}}.
\end{align*}
In fact, if \(n \geqslant 1, f \in C^{\alpha}\), for every \(\mu\) with \(| \mu | = n\) there
exists an \(x_{0} \in \TT^{d}\) such that \(\partial^{\mu} f(x_{0}) =0\), so
that
\begin{align*}
\| \partial^{\mu} f \|_{\infty} \lesssim \sup_{x \neq y} \frac{| \partial^{\mu}
f(x)- \partial^{\mu} f(y)|}{|x-y|^{\alpha^{\prime}}}.
\end{align*}
Hence, using \eqref{eqn:proof-interpolation-fractional-first} we can compute (defining \([f]_{0} = \| f \|_{\infty}\) if \(n=0\)):
\begin{align*}
[f]_{\alpha} & \lesssim \sum_{| \mu | = n} \left( \sup_{x \neq y} | \partial^{\mu} f(x)- \partial^{\mu} f(y)|\right)^{1 -
\frac{\alpha^{\prime}}{\beta^{\prime}}}\left( \sup_{x \neq y} \frac{|
\partial^{\mu} f(x)- \partial^{\mu}
f(y)|}{|x-y|^{\beta^{\prime}}}\right)^{\frac{\alpha^{\prime}}{\beta^{\prime}}} \\
& \lesssim [f]_{n}^{1 - \frac{\alpha^{\prime}}{ \beta^{\prime}}}
[f]_{\beta}^{\frac{\alpha^{\prime}}{\beta^{\prime}}}\\
& \lesssim \| f \|_{\infty}^{\big( 1 - \frac{\alpha^{\prime}}{\beta^{\prime}}
\big) \big( 1 - \frac{n}{\beta} \big)}
[f]_{\beta}^{\frac{\alpha^{\prime}}{\beta^{\prime}} +\frac{n}{\beta}\big(1 -
\frac{\alpha^{\prime}}{\beta^{\prime}}\big)}\\
& \lesssim \| f \|_{\infty}^{1 - \frac{\alpha}{\beta}}
[f]_{\beta}^{\frac{\alpha}{\beta}},
\end{align*}
which is the required result. \\
\textit{Step 2: \(k < n\).} Here we compute, following the same steps as above:
\begin{align*}
[f]_{\alpha} & \lesssim [f]_{k}^{1 - \alpha^{\prime}}
[f]_{k+1}^{\alpha^{\prime}}\\
& \lesssim \| f \|_{\infty}^{(1 - \alpha^{\prime}) \big( 1 - \frac{k}{\beta}
 \big)+ \alpha^{\prime} \big(1 - \frac{k+1}{\beta}\big) } [f]_{\beta}^{(1 -
\alpha^{\prime}) \frac{k}{ \beta} + \alpha^{\prime} \frac{k+1}{\beta} }\\
& \lesssim \| f \|_{\infty}^{1 - \frac{\alpha}{\beta}}
[f]_{\beta}^{\frac{\alpha}{\beta}},
\end{align*}
which completes the proof of the result.
\end{proof}

\begin{lemma}\label{lem:derivatives-of-the-logarithm}
Fix any \(\alpha \in (0, \infty)\) and \(f \in C^{\alpha}\) with \(f(x)>0, \ \forall x \in
\TT^{d}\) and \(\int_{\TT^{d}} f(x) \ud x = 1\).
Then defining \(m(f) = \min_{x \in \TT^{d}} f(x)\) one can bound for some
\( \overline{C}(\alpha)>0\), uniformly over $ f $:
\[ [ \log{(f)} ]_{\alpha} \leqslant \overline{C}(\alpha)
\bigg( \frac{1+[f]_{\alpha}}{m(f)} \bigg)^{\lfloor \alpha \rfloor +1}. \] 
\end{lemma}

\begin{proof}
First, observe that for any multiindex \(\mu\) with \(| \mu| = k \in \NN\) and
\(f\) sufficiently smooth we have a decomposition of the form
\begin{equation}\label{eqn:proof-regularity-logarithm-decomposition}
\partial^{\mu} \log{(f)} = \sum_{1 \leqslant p \leqslant k} \frac{
\sum_{i =1}^{\zeta(p, \mu)} C(i, p, \mu)\prod_{\lambda \in A^{i}(p,
\mu)}(\partial^{\lambda} f)}{f^{p}},
\end{equation}
where \(A^{i}(p, \mu) \subseteq \NN^{d}\) are finite sets of multiindices such
that
\[ |A^{i}(p, \mu)| \leqslant p, \qquad \text{ and } \qquad \lambda \in A^{i}(p, \mu) \Rightarrow | \lambda | \leqslant | \mu|,\] 
and
\(C(i, p, \mu) \in \RR\) are some coefficients (here \(|A^{i}(p, \mu)|\) indicates the
cardinality of the set). One can check by hand that this
decomposition holds true if \(| \mu|=1\). In addition, assuming the decomposition holds true
for some \(\mu \in \NN^{d}\), one has for any \(i \in \{1, \ldots, d\}\)
\begin{align*}
\partial_{x_{i}} \partial^{\mu} \log{(f)} = \sum_{1 \leqslant p \leqslant k} & -p\frac{
\sum_{i =1}^{\zeta(p, \mu)} C(i, p, \mu)\Big( \prod_{\lambda \in A^{i}(p,
\mu)}(\partial^{\lambda} f) \Big) \partial_{x_{i}}f }{f^{p+1}}\\
& + \frac{
\sum_{i =1}^{\zeta(p, \mu)} C(i, p, \mu) \sum_{\lambda^{\prime} \in
A^{i}(p, \mu)} \Big( \prod_{\lambda \in A^{i}(p,
\mu) \setminus \{\lambda^{\prime} \}}(\partial^{\lambda} f)\Big)
(\partial_{x_{i}} \partial^{\lambda^{\prime}}f)}{f^{p}},
\end{align*}
which is again of the required form. Hence by induction the decomposition holds
true for all \(\mu\). \\
To conclude the proof of our result we will now need the following to
inequalities. Fix any \(\alpha^{\prime} \in (0, 1), \ f,g \in C(\TT^{d})\) as
well as any smooth function \(\varphi \colon \mathcal{U} \to \RR,\)
where \(\mathcal{U} \subseteq \RR\) is an open set such that
\(f(\TT^{d}) \subseteq \mathcal{U}\). Then:
\begin{equation}\label{eqn:proof-logratihm-holder-estimates}
[ \varphi(f)]_{\alpha^{\prime}} \leqslant \sup_{x \in \TT^{d}}
|\varphi^{\prime}(f(x)) | [f]_{\alpha^{\prime}}, \qquad [f \cdot
g]_{\alpha^{\prime}} \leqslant \| f \|_{\infty} [g]_{\alpha^{\prime}} +
[f]_{\alpha^{\prime}} \| g \|_{\infty}.
\end{equation}
Both inequalities are immediate consequences of the definition of the H\"older
seminorm. For the first one:
\begin{align*}
[\varphi(f)]_{\alpha^{\prime}} = \sup_{x \neq y \in \TT^{d}} \frac{|
\varphi(f)(x) - \varphi(f)(y)|}{|x-y|^{\alpha^{\prime}}} \leqslant
\sup_{x \in \TT^{d}} | \varphi^{\prime} (f(x)) | [f]_{\alpha^{\prime}},
\end{align*}
while for the second
\begin{align*}
[f \cdot g]_{\alpha^{\prime}} \leqslant \sup_{x \neq y \in \TT^{d}}
\frac{|f(x) - f(y)||g(x)| + |g(x) -g(y)| |f(y)|}{|x - y|^{\alpha^{\prime}}}
\leqslant \| f \|_{\infty} [g]_{\alpha^{\prime}} + [f]_{\alpha^{\prime}} \| g
\|_{\infty}.
\end{align*}
Now we can complete the proof. We find
via~\eqref{eqn:proof-regularity-logarithm-decomposition} that for \(\alpha
\geqslant 1, \alpha^{\prime} = \alpha - \lfloor \alpha \rfloor\):
\begin{align*}
[ \log{(f)}]_{\alpha} & \lesssim \sum_{| \mu | = \lfloor \alpha \rfloor} \sum_{1
\leqslant p \leqslant \lfloor \alpha \rfloor} 
\sum_{i =1}^{\zeta(p, \mu)} \bigg\| \frac{\prod_{\lambda \in A^{i}(p,
\mu)}(\partial^{\lambda} f)}{f^{p}}  \bigg\|_{\infty} + \bigg[ \frac{\prod_{\lambda \in A^{i}(p,
\mu)}(\partial^{\lambda} f)}{f^{p}} \bigg]_{\alpha^{\prime}} \\
& \lesssim \sum_{| \mu | = \lfloor \alpha \rfloor} \sum_{1
\leqslant p \leqslant \lfloor \alpha \rfloor} 
\sum_{i =1}^{\zeta(p, \mu)} \frac{ [f]_{\alpha}^{|A^{i}(p,
\mu)|}}{m(f)^{p}} + \frac{[f]_{\alpha}^{|A^{i}(p, \mu)|+1}}{m(f)^{p+1}}.
\end{align*}
where in the last step we used~\eqref{eqn:proof-logratihm-holder-estimates}.
Now, since \(\int f(x) \ud x =1\) we have that \(m(f) \leqslant 1\). In
addition we have that \(|A^{i}(p, \mu)| \leqslant \lfloor \alpha \rfloor\), so
overall we have that:
\begin{align*}
[ \log{(f)}]_{\alpha} & \lesssim \left(  \frac{1 + [f]_{\alpha}}{m(f)}
\right)^{\lfloor \alpha \rfloor + 1},
\end{align*}
which is the required inequality. The case \(\alpha \in (0,1)\) is much simpler
and follows directly from~\eqref{eqn:proof-logratihm-holder-estimates}.
\end{proof}

\end{document}